\documentclass{article}
\usepackage{cite,verbatim}
\usepackage{mathtools}
\usepackage{amssymb,amsmath,mathrsfs,graphicx}
\usepackage{fancyhdr}
\usepackage{float}
\usepackage{tikz}
\usetikzlibrary{matrix,arrows}
\usepackage{graphicx}
\usepackage{epstopdf}
\newtheorem{precor}{{\bf Corollary}}

\newenvironment{cor}{\begin{precor}{\hspace{-0.5
em}{\bf.\ }}}{\end{precor}}
\newtheorem{precon}{{\bf Conjecture}}

\newtheorem{predefin}{{\bf Definition}}

\newenvironment{defin}[1]{\begin{predefin}{\hspace{-0.5
em}{\bf.\ }}{\rm
#1}\hfill{$\blacktriangleleft$}}{\end{predefin}}
\newtheorem{preexm}{{\bf Example}}

\newenvironment{exm}[1]{\begin{preexm}{\hspace{-0.5
em}{\bf.\ }}{\rm #1}\hfill{$\blacktriangleright$}}{\end{preexm}}

\newtheorem{preappl}{{\bf Application}}

\newtheorem{prelem}{{\bf Lemma}}

\newenvironment{lem}{\begin{prelem}{\hspace{-0.5
em}{\bf.\ }}}{\end{prelem}}
\newtheorem{preproof}{{\bf Proof.\ }}

\newenvironment{proof}[1]{\begin{preproof}{\rm
#1}\hfill{$\blacksquare$}}{\end{preproof}}
\newtheorem{presproof}{{\bf Sketch of Proof.\ }}

\newtheorem{prethm}{{\bf Theorem}}

\newenvironment{thm}{\begin{prethm}{\hspace{-0.5
em}{\bf.\ }}}{\end{prethm}}
\newtheorem{prealphthm}{{\bf Theorem}}

\newenvironment{alphthm}{\begin{prealphthm}{\hspace{-0.5
em}{\bf.\ }}}{\end{prealphthm}}
\newtheorem{prepro}{{\bf Proposition}}

\newenvironment{pro}{\begin{prepro}{\hspace{-0.5
em}{\bf.\ }}}{\end{prepro}}
\newtheorem{preprb}{{\bf Problem}}

\newenvironment{prb}{\begin{preprb}{\hspace{-0.5
em}{\bf.\ }}}{\end{preprb}}

\def\conct[#1,#2]{\mbox {${#1} \leftrightarrow {#2}$}}
\def\dconct[#1,#2]{\mbox {${#1} \rightarrow {#2}$}}
\def\deg[#1,#2]{\mbox {$d_{_{#1}}(#2)$}}
\def\mindeg[#1]{\mbox {$\delta_{_{#1}}$}}
\def\maxdeg[#1]{\mbox {$\Delta_{_{#1}}$}}
\def\outdeg[#1,#2]{\mbox {$d_{_{#1}}^{^+}(#2)$}}
\def\minoutdeg[#1]{\mbox {$\delta_{_{#1}}^{^+}$}}
\def\maxoutdeg[#1]{\mbox {$\Delta_{_{#1}}^{^+}$}}
\def\indeg[#1,#2]{\mbox {$d_{_{#1}}^{^-}(#2)$}}
\def\minindeg[#1]{\mbox {$\delta_{_{#1}}^{^-}$}}
\def\maxindeg[#1]{\mbox {$\Delta_{_{#1}}^{^-}$}}
\def\isdef{\mbox {$\ \stackrel{\rm def}{=} \ $}}
\def\dre[#1,#2,#3]{\mbox {${\cal E}_{_{#3}}(#1,#2)$}}
\def\pdre[#1,#2,#3]{\mbox {${\cal P}_{_{#3}}(#1,#2)$}}
\def\var[#1,#2]{\mbox {${\rm Var}_{_{#1}}(#2)$}}
\def\ls[#1]{\mbox {$\xi^{^{#1}}$}}
\def\onvhom[#1,#2]{\mbox {${\rm Hom^{v}}(#1,#2)$}}
\def\onehom[#1,#2]{\mbox {${\rm Hom^{e}}(#1,#2)$}}
\def\core[#1]{\mbox {$#1^{^{\bullet}}$}}
\def\cay[#1,#2]{\mbox {${\rm Cay}({#1},{#2})$}}
\def\cays[#1,#2]{\mbox {${\rm Cay_{s}}({#1},{#2})$}}
\def\dirc[#1]{\mbox {$\stackrel{\rightarrow}{C}_{_{#1}}$}}
\def\cycl[#1]{\mbox {${\bf Z}_{_{#1}}$}}

\def\sdg[#1]{\mbox {$\stackrel{\leftrightarrow}{#1}$}}

\def\go{\matr{g}_{_{odd}}}

\def\ii{\mathsf{v}}
\def\oo{\mathsf{u}}

\long\def\Perm[#1]{\mbox{$\lfloor #1 \rceil$}}
\long\def\nat[#1]{\mbox{$\widehat{1...#1}$}}

\newcommand{\matr}[1]{\mathrm{#1}} % For matrices, graphs, ... (structured objects)
 \def\authora{Amir Daneshgar\footnote{Correspondence should be addressed to {\tt daneshgar@sharif.ir}.}}
\def\authorb{Meysam Madani}
\def\addressa{Department of Mathematical Sciences\\
Sharif University of Technology\\
{P.O. Box {\rm 11155--9415}} \\
Tehran, Iran. \\}
\author{
{ \authora }\\
{ \authorb }\\[1mm]
\addressa
}

 \title{On the odd girth and the circular chromatic number of generalized Petersen graphs}

 \date{\today}
\begin{document}
\maketitle

 \begin{abstract}{
A class of simple graphs such as ${\cal G}$ is said to be {\it odd-girth-closed}
if for any positive integer $g$ there exists a
graph $\matr{G} \in {\cal G}$ such that the odd-girth of $\matr{G}$ is greater than or equal to $g$.
An odd-girth-closed class of graphs ${\cal G}$ is said to be {\it odd-pentagonal}
if there exists a positive integer $g^*$
depending on ${\cal G}$ such that any graph $\matr{G} \in {\cal G}$ whose odd-girth is greater than
$g^*$ admits a homomorphism to the five cycle (i.e. is $\matr{C}_{_{5}}$-colorable).

 In this article, we show that finding the odd girth of generalized Petersen graphs can be transformed to an integer programming problem, and using this we explicitly compute the odd girth of such graphs, showing that the class is odd-girth-closed.
Also, motivated by showing that the class of generalized Petersen graphs is odd-pentagonal, we study the circular chromatic number of such graphs.
\\[.4cm]
}\end{abstract}

 \section{Introduction}

 Let us call a class of simple graphs ${\cal G}$ {\it girth-closed} (resp. {\it odd-girth-closed})
if for any positive integer $g$ there exists a
graph $\matr{G} \in {\cal G}$ such that the girth (resp. odd-girth) of $\matr{G}$ is greater than or equal to $g$.
A girth-closed (resp. odd-girth-closed) class of graphs ${\cal G}$ is said to be {\it pentagonal} (resp. {\it odd-pentagonal})
if there exists a positive integer $g^*$
depending on ${\cal G}$ such that any graph $\matr{G} \in {\cal G}$ whose girth (resp. odd-girth) is greater than
$g^*$ admits a homomorphism to the five cycle (i.e. is $\matr{C}_{_{5}}$-colorable).
 The following question of J.~Ne\v{s}et\v{r}il has been the main motivation for a number of contributions in graph theory.

 \begin{prb}{\rm \cite{Nes}}\label{prb:pentagon}
Is the class of simple $3$-regular graphs pentagonal?
\end{prb}

 Note that every simple $3$-regular graph except $\matr{K}_{_{4}}$ admits a homomorphism to $\matr{K}_{_{3}} \simeq \matr{C}_{_{3}} $ by Brooks' theorem. On the other hand, it is quite interesting to note that the answer is negative if we ask the same question with the five cycle $\matr{C}_{_{5}}$ replaced by $\matr{C}_{_{7}}, \matr{C}_{_{9}}$ or $\matr{C}_{_{11}}$ (see \cite{CAT88,ghebleh,HAT05,KNS01,WW01} for this and the background on other negative results).

 A couple of relaxations of Problem~\ref{prb:pentagon} have already been considered in the literature.
The following relaxation has been attributed to L.Goddyn (see \cite{ghebleh}) and to best of our knowledge is still open.

 \begin{prb} Is it true that every cubic graph with sufficiently large girth has circular chromatic number strictly less than $3$?
\end{prb}

 Also, M.~Ghebleh introduced the following relaxation of Problem~\ref{prb:pentagon} and answered it negatively by
constructing the class of spiderweb graphs whose circular chromatic numbers are equal to $3$ (see \cite{ghebleh} for more on
spiderweb graphs).

 \begin{prb}{\rm \cite{ghebleh}}
Is the class of simple $3$-regular graphs odd-pentagonal?
\end{prb}

 Although, to answer Problem~\ref{prb:pentagon} negatively it is sufficient to introduce a subclass of simple $3$-regular graphs
which is not pentagonal, it is still interesting to study pentagonal subclasses of $3$-regular graphs, even if the problem has a negative answer, since this study will definitely reveal structural properties that can influence the graph homomorphism problem.

 Clearly, to analyze the above mentioned problems one has to have a good understanding about the cycle structure of the graphs, where it seems that the answers are strongly related to the number theoretic properties of the cycle lengths of the graphs being studied.

 This article may be classified into the positive part of the above scenario which is related to the study of pentagonal
and odd-pentagonal subclasses of $3$-regular graphs on the one hand, and the study of subclasses of $3$-regular graphs with
circular chromatic number less than $3$ on the other\footnote{See \cite{ghebleh} and references therein for other related results and the background. Also see \cite{PZ02} for a similar approach.}.
Also, note that other variants of the positive scenario have already been studied either using stronger average degree conditions or topological conditions on the sparse graphs. For instance, we have 

\begin{alphthm}
{\rm \cite{oddclose}}  The class of simple graphs as $\matr{G}$ for which every subgraph of $\matr{G}$ has average degree less than $12/5$, is pentagonal $($actually with $g^*=3)$.
\end{alphthm}
Also, using results of \cite{close} we deduce that,

\begin{alphthm}
{\rm \cite{close}}  The class of planar graphs, projective planar graphs, graphs that can be embedded on the torus or Klein bottle are pentagonal.
\end{alphthm}

Along the same lines we also have,

\begin{alphthm}
{\rm \cite{GGH01}}  For every fixed simple graph $\matr{H}$ the class of $\matr{H}$-minor free graphs is pentagonal.
\end{alphthm}

 Our major motivation to study the cycle structure of generalized Petersen graphs was, of course,
related to our belief that the class of generalized Petersen graphs is odd-pentagonal. Needless to say, some results and techniques used in the sequel are related to the specific properties of generalized Petersen graphs and may be of independent interest.

 In the rest of this section we go through a couple of preliminary definitions and concepts and after that we will
explain the main results of this article.
In subsequent sections, first, in Section~\ref{sec:oddgirth} we explicitly compute the odd-girth of generalized Petersen graphs and it will follow that the class is odd-girth closed (while it is not girth-closed). Then we will use the data to study the circular chromatic number of these graphs in Section~\ref{sec:circhrom} and will find partial evidence for a positive answer to our motivating question.

 \subsection{Preliminaries}

 In this article,
for two integers $m \leq n$, the notation $m|n$ indicates that $m$ divides $n$,
the greatest common divisor is denoted by $\gcd(m,n)$ and also, we define
$$[m,n] \isdef \{m,m+1,\cdots,n\}.$$
The size of any finite set $A$ is denoted by $|A|$.

 Hereafter, we only consider finite simple graphs as $\matr{G}=(V(\matr{G}),E(\matr{G}))$, where $V(\matr{G})$ is the vertex set and $E(\matr{G})$ is the edge set. An edge with end vertices $u$ and $v$ is denoted by $uv$. Moreover,
for any graph $\matr{G}$, the odd girth $\go (\matr{G})$ is the length of the shortest odd cycle of $\matr{G}$.
The complete graph on $n$ vertices and the cycle of length $n$ are denoted by $\matr{K}_{_{n}}$ and $\matr{C}_{_{n}}$, respectively (up to isomorphism).

 Throughout the article, a homomorphism $f: \matr{G} \longrightarrow \matr{H}$ from a graph $\matr{G}$ to a
graph $\matr{H}$ is a map $f: V(\matr{G}) \longrightarrow V(\matr{H})$ such that $uv
\in E(\matr{G})$ implies $f(u)f(v) \in E(\matr{H})$.
If there exists a homomorphism from $\matr{G}$ to
$\matr{H}$ then we may simply write $\matr{G} \longrightarrow \matr{H}$. (For more on graph homomorphisms and their central role
in graph theory see \cite{HN04}.)

 Let $k$ and $d$ be positive integers such that $n \geq 2d$. Then the
{\em circular complete graph} $\matr{K}{_{\frac{n}{d}}}$ is the graph with
the vertex set $\{0, 1, \cdots, n-1\}$ in which $i$ is connected to
$ j$ if and only if $d \leq |i-j| \leq n-d$. A graph $\matr{G}$ is said to
be $(n,d)$-colourable if $\matr{G}$ admits a homomorphism to
$\matr{K}_{_{\frac{n}{d}}}$. The {\em circular chromatic number} $\chi{_{_c}}(\matr{G})$
of a graph $\matr{G}$ is the minimum of those ratios $\frac{n}{d}$ for
which $\gcd(n,d)=1$ and $\matr{G}$ admits a homomorphism to
$K_{_{\frac{n}{d}}}$ (see \cite{vince,ZHU01,ZHU06} for more on circular chromatic number and its properties).
Also, note that $\matr{C}_{_{2k+1}} \simeq \matr{K}_{_{\frac{2k+1}{k}}},$
and consequently $$\chi_{_c}(\matr{C}_{_{2k+1}})=2+\frac{1}{k}.$$

 For a graph $\matr{G}$ the graph $\matr{G}^{^{\frac{1}{d}}}$ is defined as the graph obtained from $\matr{G}$ by replacing each edge by a path of length $d$ (i.e. subdividing each edge by $d-1$ vertices). The $r$th power of a graph $\matr{G}$,
denoted by $\matr{G}^{r}$,
is a graph on the vertex set $V(\matr{G})$, in which two vertices are connected by an edge if there exist an $r$-walk between them in $\matr{G}$. Also, the fractional power of a graph is defined as
$$\matr{G}^{\frac{r}{d}}\isdef \left(\matr{G}^{\frac{1}{d}}\right)^{r}.$$
Note that for simple graphs, $\matr{G} \longrightarrow \matr{H}$ implies that $\matr{G}^{r} \longrightarrow \matr{H}^{r}$
for any positive integer $r > 0$. Hence, Problem~\ref{prb:pentagon} is closely related to the study of the chromatic number of the third power of sparse $3$-regular simple graphs (see \cite{DH08,haji10,haji11}).

 The Petersen graph is an icon in graph theory. The generalized Petersen graphs are introduced in \cite{cox}
and given the name along with a standard notation by M. Watkins in \cite{watkins}.
The generalized Petersen graph, $\matr{Pet}(n,k)$, is defined as follows.
\begin{defin}{\label{def:Petersen}
In Watkins' notation the generalized Petersen graph $\matr{Pet}(n,k)$ for positive integers $n$ and $k$, where $2< 2k\leq n$, is a graph on $2n$ vertices, defined as follows,
$$V(\matr{Pet}(n,k))\isdef \{u_{_0}, u_{_1}, \dots , u_{_{n-1}}\} \cup \{v_{_0}, v_{_1}, \dots , v_{_{n-1}}\},$$
$$E(\matr{Pet}(n,k))\isdef \left(\displaystyle{\bigcup_{i=0}^{n-1} \{u_{_i}u_{_{i+1}}\}}\right) \cup
\left(\displaystyle{\bigcup_{i=0}^{n-1} \{u_{_i}v_{_{i}}\} }\right) \cup
\left(\displaystyle{\bigcup_{i=0}^{n-1} \{v_{_i}v_{_{i+k}}\} }\right),$$
where $+$ stands for addition modulo $n$ (see Figure \ref{pic:watkins}).
}\end{defin}
\begin{figure}[ht]
\includegraphics[width=6.5cm]{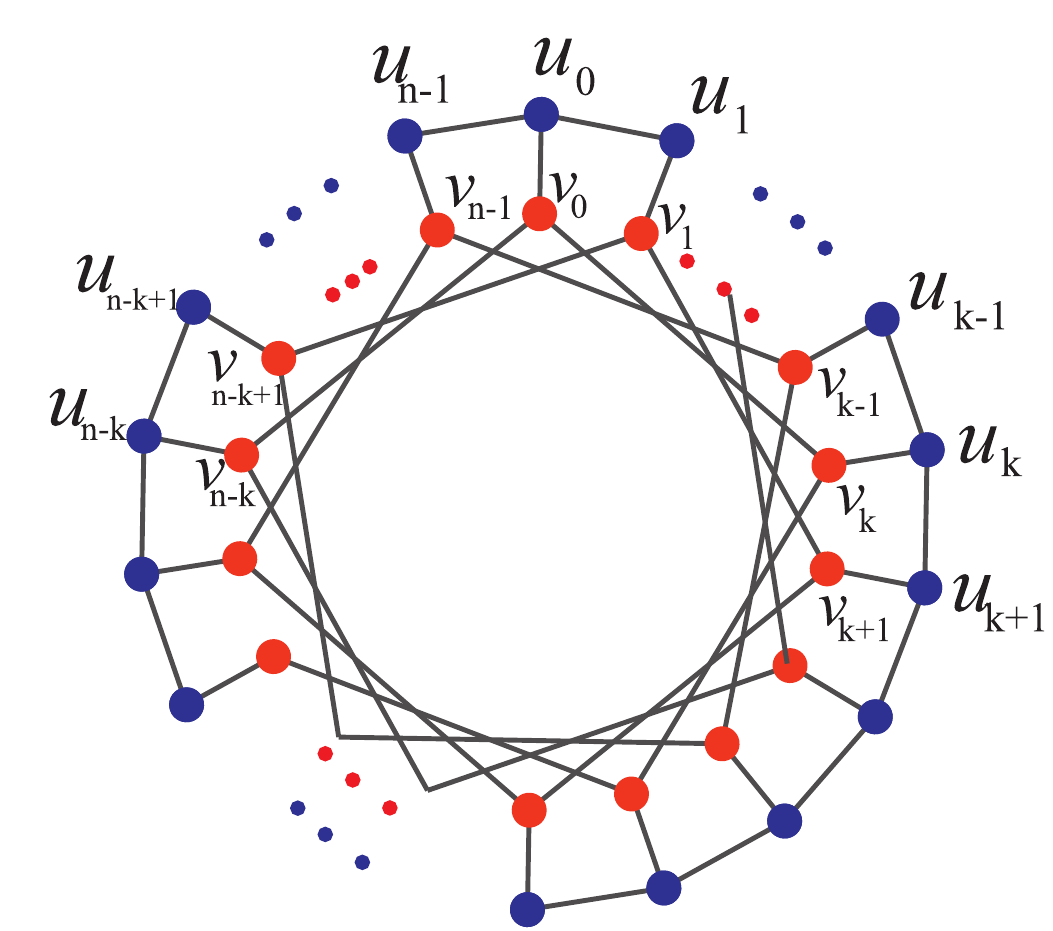}
\centering{\caption{Watkins' notation for generalized Petersen graph $\matr{Pet}(n,k)$.
}\label{pic:watkins}}
\end{figure}
Some basic properties of generalized Petersen graphs are as follows (e.g. see \cite{fox}).
\begin{itemize}
\item Except $\matr{Pet}(2k,k)$, all generalized Petersen graphs are $3$-regular.
\item $\matr{Pet}(n,k)$ is bipartite if and only if $n$ is even and $k$ is odd, otherwise $\matr{Pet}(n,k)$ is $3$-chromatic.
\item $\matr{Pet}(n,k)$ is isomorphic to $\matr{Pet}(n,m)$, if and only if either $m\not\stackrel{n}{\equiv}\pm k$ or $mk \stackrel{n}{\equiv} \pm 1$ \cite{iso}.
\item $\matr{Pet}(n,k)$ is vertex transitive if and only if $(n,k)=(10,2)$ or $k^2\stackrel{n}{\equiv} \pm 1$ \cite{frucht}.
\item $\matr{Pet}(n, k)$ is edge-transitive (see \cite{frucht}) if and only if 
$$(n, k)\in \{(4, 1), (5, 2), (8, 3), (10, 2), (10, 3), (12, 5), (24, 5)\}.$$
\item $\matr{Pet}(n, k)$ is a Cayley graph (see \cite{cayley}) if and only if $k^2\stackrel{n}{\equiv} 1$.
\item Every generalized Petersen graph is a unit distance graph \cite{unit} (a graph formed from a collection of points in the Euclidean plane by connecting two points by an edge whenever the distance between the two points is exactly one).
\item Some specially named generalized Petersen graphs are the {\it Petersen graph} $\matr{Pet}(5,2)$, the \textit{D$\ddot{u}$rer graph} $\matr{Pet}(6,2)$,
the \textit{M$\ddot{o}$bius-Kantor graph} $\matr{Pet}(8,3)$, the \textit{dodecahedron} $\matr{Pet}(10,2)$, the \textit{Desargues graph} $\matr{Pet}(10,3)$ and the \textit{Nauru graph} $\matr{Pet}(12,5)$.
\end{itemize}
We know when generalized Petersen graphs are Hamiltonian and we also know about their automorphism groups (the interested reader is referred to the exiting literature for more on these graphs).
\subsection{Main results}
\indent
Our main results in this article can be divided into two categories. The first category is concerned about the
odd girth of generalized Petersen graphs, discussed in Section~\ref{sec:oddgirth}, and the following theorem presents an explicit description of this parameter as
our first main result.
\begin{thm}{\label{thm:oddgirth}
Let $\mathbb{O}$ be the set of odd numbers, and define,
$$\matr{Ind}(n,k)=\left\lbrace
\begin{array}{ll} \lbrace t\ | \ t\ {\rm is\ odd},\ \ 0\leq t\leq \min \{\frac{2k}{\gcd(n,k)}, \lfloor\frac{(k-1)^2}{n}\rfloor+1 \}\rbrace & k\ {\rm and} \ n \ {\rm are \ odd},\\
\lbrace t\ | \ 0<t\leq \min \{\frac{2k}{\gcd(n,k)}, \lfloor\frac{(k-1)^2}{n}\rfloor\}\rbrace
& \frac{n}{\gcd(n,k)} \ {\rm is \ odd\ and}\ k\ {\rm is\ even},\\
\lbrace t\ | \ 0<t\leq \min \{\frac{k}{\gcd(n,k)}, \lfloor\frac{(k-1)^2}{n}\rfloor\}\rbrace & \frac{n}{\gcd(n,k)} \ {\rm and}\ k\ {\rm are\ even}.
\end{array} \right. $$
and
$${\cal G} \isdef \displaystyle{\bigcup_{t\in \matr{Ind}(n,k)}} \left\lbrace tn+(1-k)\lfloor \frac{tn}{k} \rfloor+2 , \ (1+k)\lceil \frac{tn}{k}\rceil -tn+2\right\rbrace.$$
Then we have
$$\go(\matr{Pet}(n,k))= \min \left(\left(\left\lbrace \frac{n}{\gcd(n,k)},\ k+3\right\rbrace \cup {\cal G}\right)
\cap \mathbb{O}\right).$$
}\end{thm}
Our main strategy to obtain the above formulation is based on the key observation that
the problem of finding the odd girth of $\matr{Pet}(n,k)$ can be reduced to finding the solutions of the following integer program,
$$
\begin{array}{llll}
\min & & \oo+ \ii_{_+}+\ii_{_-} &\\
& & \oo + k(\ii_{_+}-\ii_{_-}) & =tn, \\
& &\oo+\ii_{_+}+\ii_{_-} &=2r+1, \\
& &r, \oo, \ii_{_+}, \ii_{_-}& \geq 0,\\
& &t \in \mathbb{Z}.&
\end{array}\qquad (*)
$$
We will see that the optimization problem $(*)$ may have some trivial solutions. In this regard, a solution
$(\oo, \ii_{_+}, \ii_{_-})$
of $(*)$ is said to be a {\it trivial} solution if either $\oo=0$ or $ \ii_{_+}+ \ii_{_-}=0$. Evidently, other solutions are referred to as {\it nontrivial} solutions.
Moreover, note that $r$ and $t$ are uniquely determined by a solution $(\oo, \ii_{_+}, \ii_{_-})$. Hence, sometimes, when there is no ambiguity, we may talk about a solution $(\oo, \ii_{_+}, \ii_{_-})$ or a solution $(\oo, \ii_{_+}, \ii_{_-},r,t)$, when we want to explicitly refer to parameters $r$ and $t$.
We may also say that $(\oo, \ii_{_+}, \ii_{_-},r,t)$ is {\it feasible} if the parameters satisfy conditions of $(*)$.
Note that $(*)$ has a solution if and only if its {\it feasible set} is non-empty.

 Using Theorem~\ref{thm:oddgirth} we may discuss some asymptotic properties of the odd girth of generalized Petersen graphs.
In this regard we prove,
\begin{thm}{\label{thm:girthbounds}
The odd girth of a non-bipartite generalized Petersen graph, $\matr{Pet}(n,k)$, is either equal to $k+3$ or satisfies the following inequality
$$\max\left(\displaystyle{\frac{n}{k}},\left(\min\{\gcd(n,k-1),\gcd(n,k+1)\}+2\right)\right) \leq \go(\matr{Pet}(n,k)) \leq \displaystyle{\frac{n}{k}}par(k)+k+1,$$
where $par(k)$ is the parity function which is equal to one when $k$ is odd and is
zero otherwise.
}\end{thm}

 Note that by Theorem~\ref{thm:girthbounds} the odd girth of $\matr{Pet}(n,k)$ tends to infinity either if both
$\frac{n}{k}$ and $k+3$ tend to infinity, or if $\min\{\gcd(n,k-1), \gcd(n,k+1)\}$ tends to infinity.
However, the converse may not be true.

For instance, for fixed and odd $k$, the odd girth of $\matr{Pet}(jk,k)$
tends to infinity, when $j$ is odd and tends to infinity. On the other hand, for a fixed even integer $j$, 
the odd girth of $\matr{Pet}(jk,k)$ by Theorem~\ref{thm:oddgirth} is equal to $k+3$ and tends to infinity, when $k$ is even and tends to infinity. Note that in this case $\frac{n}{k}$ is fixed.
Both of these observations show that, $\go(\matr{Pet}(n,k))$ may tend to infinity while one of the parameters
$\frac{n}{k}$ or $k$ is fixed.

Moreover, note that for $n=(k-1)(k+1)+1$, if $k$ tends to infinity, the odd girth $\go(\matr{Pet}(n,k))$ tends to infinity
since both $k$ and $\frac{n}{k}$ tend to infinity, while in this case  $\min\{\gcd(n,k-1),\gcd(n,k+1)\}=1$ is fixed.

 It is easy to verify that in Watkins' notation $u_{_{0}}u_{_{1}}v_{_{1}}v_{_{k+1}}u_{_{k+1}}u_{_{k}}v_{_{k}}v_{_{0}}u_{_{0}}$ is a closed walk in $\matr{Pet}(n,k)$, implying that the girth of generalized Petersen graphs are always less than or equal to $8$. On the other hand, by Theorem~\ref{thm:girthbounds} generalized Petersen graphs are odd-girth closed.

 \begin{cor}
The class of generalized Petersen graphs are odd-girth closed but not girth-closed.
\end{cor}

The second category of our results is concerned with
circular chromatic number of generalized Petersen graphs, discussed in Section~\ref{sec:circhrom}.
The key observation in this regard is the fact that the information we already extract about the odd cycles of $\matr{Pet}(n,k)$
will help to prove bounds on the clique-number of some powers of these graphs, or helps to prove existence and nonexistence results on homomorphisms to powers of cycles, which in turn will give rise to
some bounds on the circular chromatic number of generalized Petersen graphs.

 Our first result in this regard is the following theorem.
\begin{thm}{\label{thm:homk2ko} \
\begin{itemize}
\item[{\rm (a)}]
Suppose that the system $(*)$ has no trivial solution. Then, if $(\oo, \ii_{_+},\ii_{_-},r)$ is a (non-trivial) solution, then
$$\matr{K}_{_{4r+2}} \rightarrow \matr{Pet}(n,k)^{^{2r+1}}.$$
\item[{\rm (b)}]
Let $\overline{\matr{K}}_{_{\frac{2n}{4k+4}}}$ be the complement of the circular complete graph $\matr{K}_{_{\frac{2n}{4k+4}}}$. Then, for $n>2k$, we have
$$\overline{\matr{K}}_{_{\frac{2n}{4k+4}}} \rightarrow \matr{Pet}(n,2k)^{^{2k+1}}.$$
\end{itemize}
}\end{thm}

 Using Theorem~\ref{thm:homk2ko} we will obtain the following corollary on lower bounds for the circular chromatic number of Petersen graphs.
\begin{cor}{\label{cor:compk}
\begin{itemize}
\item[{\rm (a)}]
Let $n$ and $k$ be odd and suppose that the system $(*)$ has no trivial solution.
If $(\oo, \ii_{_+},\ii_{_-},r)$ is a (non-trivial) solution of $(*)$, then
$$2+ \dfrac{4r}{4r^2+2r+1} \leq \chi_{_c}(\matr{Pet}(n, k)).$$
\item[{\rm (b)}]
If $\go(\matr{Pet}(n,2k))=2k+3$, we have
$$ \frac{2n(2k+1)}{2kn+\lfloor \frac{2n}{4k+2}\rfloor} \leq \chi_{_c}(\matr{Pet}(n,2k)).$$
\end{itemize}
}\end{cor}

 A lower bound for circular chromatic number of Petersen graphs has already been introduced as follows.
\begin{alphthm}{{\rm \cite{ghebleh}}\label{pro:gheb} For any $n>2k$, we have
$$2+\frac{2}{2k+1} \leq \chi_{_c}(\matr{Pet}(n,2k)).$$
}\end{alphthm}

 In Proposition~\ref{pro:circular} (Section~\ref{sec:circhrom}) we will show that the lower bound
of Corollary~\ref{cor:compk}$(b)$ can be strictly greater
than the lower bound obtained in \cite{ghebleh}.

 Now, let us concentrate on the upper bounds. For this we define the graph $\matr{Pb}(n,k)$
on the cycle $\matr{C}_{_n}$ by connecting edges having distance $k$
(i.e. the graph is obtained by identifying vertices $v_{_i}$ and $u_{_i}$ in $\matr{Pet}(n,k)$).
Also, let $\matr{C}_{_n}^{^r}$ be the $r$-th power of $\matr{C}_{_n}$.
\begin{thm}{\label{thm:uphom}\
\begin{itemize}
\item[{\rm (a)}]
If $n$ is odd, $k$ is even, $n\stackrel{k-1}{\equiv} \pm 2$ and $s=\frac{(n-4)(k-2)}{2(k-1)}$, then there exists a homomorphism
$$\sigma : \matr{Pb}(n,k) \rightarrow \matr{K}_{_{\frac{n}{s}}}.$$
\item[{\rm (b)}]
For any generalized Petersen graph $\matr{Pet}(n,k)$, there exists a homomorphism
$$\sigma: \matr{Pet}(n,k) \rightarrow \matr{Pb}(n,k).$$
\item[{\rm (c)}] If $n$ and $k$ are odd and $n>2k+1$, there exists a homomorphism
$$\sigma:\ \matr{Pet}(n,k) \rightarrow \matr{C}_{_{n}}^{^k}.$$
\end{itemize}
}\end{thm}

 Using Theorem~\ref{thm:uphom} we will obtain the following corollary on upper bounds for the circular chromatic number of generalized Petersen graphs.
\begin{cor}{\label{cor:upbound}\
\begin{itemize}
\item[{\rm (a)}]
Let $k$ be even, $n$ be odd and $n\stackrel{k-1}{\equiv}\pm 2$. Then
$$\chi_{_c}(\matr{Pet}(n,k))\leq \frac{2n(k-1)}{(n-4)(k-2)}.$$
\item[{\rm (b)}]
If $n$ and $k$ are odd and $n>2k+1$, then
$$\chi_{_c}(\matr{Pet}(n,k))\leq \frac{2n}{n-k}.$$
\end{itemize}
}\end{cor}

 As a direct corollary one may conclude the following on odd-pentagonality of some subclasses
of the class of generalized Petersen graphs.
\begin{cor}{\label{cor:ogirthclosed}
Let ${\cal C}$ be the subclass of the class of generalized Petersen graphs for which at least one of the following conditions
hold.
\begin{itemize}
\item[{\rm (a)}] $\matr{Pet}(n,k)$, where $k$ is even, $n$ is odd and $n\stackrel{k-1}{\equiv}\pm 2$.
\item[{\rm (b)}] $\matr{Pet}(n,k)$, where both $n$ and $k$ are odd and $n\geq 5k$.
\end{itemize}
Then ${\cal C}$ is odd-pentagonal.
}\end{cor}

 It is instructive to note that by results of \cite{iso} $\matr{Pet}(n,k)$ is isomorphic to $\matr{Pet}(n,m)$, if and only if either $m\not\stackrel{n}{\equiv}\pm k$ or $mk \stackrel{n}{\equiv} \pm 1$. Hence, one may construct larger
odd-pentagonal subclasses of the class of generalized Petersen graphs. Also, note that this study motivates the interesting problem of characterizing conditions under which we have $$\matr{Pet}(n,k) \longrightarrow \matr{Pet}(n',k').$$

 \section{Odd girth of Generalized Petersen Graph}\label{sec:oddgirth}
In this section, we will prove Theorems~\ref{thm:oddgirth} and \ref{thm:girthbounds} by analyzing the solutions of the integer program $(*)$.
To start, let $\matr{C}$ be a cycle of length $\ell$ in $\matr{Pet}(n,k)$ with a fixed orientation as
$$\matr{C}=w_{_0}w_{_1}w_{_2}\ldots w_{_\ell}w_{_0}.$$
We define,
\begin{itemize}
\item $\oo_{_+}(\matr{C}) \isdef |\{ w_{_j}w_{_{j+1}} \ |\ \exists i \in [0,n-1], \ (w_{_j}, w_{_{j+1}})
= (u_{_i},u_{_{i+1}})\}|$,\\
\item $\oo_{_-}(\matr{C}) \isdef |\{ w_{_j}w_{_{j+1}} \ |\ \exists i \in [0,n-1], \ (w_{_j}, w_{_{j+1}})= (u_{_i},u_{_{i-1}}) \}|$,\\
\item $\ii_{_+}(\matr{C}) \isdef |\{ w_{_j}w_{_{j+1}} \ |\ \exists i \in [0,n-1], \ (w_{_j}, w_{_{j+1}})= (v_{_i},v_{_{i+k}}) \}|$,\\
\item $\ii_{_-}(\matr{C}) \isdef |\{ w_{_j}w_{_{j+1}} \ |\ \exists i \in [0,n-1], \ (w_{_j}, w_{_{j+1}})= (v_{_i},v_{_{i-k}}) \}|$,\\
\item $b(\matr{C}) \isdef |\{ w_{_j}w_{_{j+1}} \ |\ \exists i \in [0,n-1], \ \{w_{_j}, w_{_{j+1}}\}= \{ u_{_i},v_{_{i}}\} \}|$
\end{itemize}
in which addition, $+$, and subtraction,$-$, are in $\mathbb{Z}_{_{n}}$. Also, hereafter, when we refer to any one of the above parameters we assume that the prefixed orientation is clear from the context.
Let us list the following basic facts for further reference.
\begin{lem}{\label{lem:cyc-odd} For any cycle $\matr{C}$ of length $\ell$ $($along with a fixed orientation$)$ in
$\matr{Pet}(n,k)$, we have,
\begin{itemize}
\item[{\rm (a)}] Parameters
$\oo_{_+}(\matr{C}),\oo_{_-}(\matr{C}),\ii_{_+}(\matr{C}),\ii_{_-}(\matr{C}),$ and $b(\matr{C})$ are all non-negative.
\item[{\rm (b)}] $\oo_{_+}(\matr{C})+\oo_{_-}(\matr{C})+\ii_{_+}(\matr{C})+\ii_{_-}(\matr{C})+b(\matr{C})=\ell.$
\item[{\rm (c)}]
$\oo_{_+}(\matr{C})+\oo_{_-}(\matr{C}) \leq n \quad {\rm and} \quad \ii_{_+}(\matr{C})+\ii_{_-}(\matr{C}) \leq n.$
\item[{\rm (d)}]
$\oo_{_+}(\matr{C})-\oo_{_-}(\matr{C})+k(\ii_{_+}(\matr{C})-\ii_{_-}(\matr{C}))\stackrel{n}{\equiv} 0.$
\item[{\rm (e)}] $b(\matr{C})$ is an even number. Also, if $\oo_{_+}(\matr{C})+\oo_{_-}(\matr{C})\neq 0$ and $\ii_{_+}(\matr{C})+\ii_{_-}(\matr{C})\neq 0$, then $b(\matr{C})\geq 2$.
\end{itemize}
}\end{lem}
\begin{proof}{Statements $(a), (b)$ and $(c)$ are clear by definitions. For $(d)$, let
$$\matr{C}=w_{_0}w_{_1}w_{_2}\ldots w_{_\ell}w_{_0}$$
and consider the potential function $\varsigma$ on the vertex set of $\matr{Pet}(n,k)$ to $[0,n-1]$, that is equal to $i$ exactly on
$u_{_i}$ and $v_{_i}$ (in Watkins' notation). Without loss of generality, let $\varsigma(w_{_0})=0$. Note that
$$0=\displaystyle{\sum_{\matr{C}} (\varsigma(w_{_{j+1}})-\varsigma(w_{_j}))} \stackrel{n}{\equiv}
(\oo_{_+}(\matr{C})+k\ii_{_+}(\matr{C}))-(\oo_{_-}(\matr{C})+k\ii_{_-}(\matr{C})).$$

 Statement $(e)$ is also clear by considering the fact that the number of transitions between the parts consisting of ${u_{_{i}}}$'s and ${v_{_{i}}}$'s is always an even number. Also, if the cycle $\matr{C}$ has some vertices in both parts,
then $b(\matr{C}) \not = 0$.
}\end{proof}

 Now, we consider some properties of the solution of the integer program $(*)$.

 \begin{lem}{\label{lem:basic} If $(\oo,\ii_{_+},\ii_{_-})$ is a solution of $(*)$, then
\begin{itemize}
\item[{\rm (a)}] Either $\ii_{_-}=0$ or $\ii_{_+}=0$.
\item[{\rm (b)}] If $n$ is even, then $k$ is even.
\item[{\rm (c)}] $ 0 \leq \oo+\ii_{_-}+\ii_{_+}\leq n$.
\item[{\rm (d)}] If $k$ is odd, then $\oo< k$.
\item[{\rm (e)}] If $k$ is even, then either $(\oo,\ii_{_+},\ii_{_-},t)=(k,0,1,0)$ or $\oo<k$.
\item[{\rm (f)}] $(n,k)|\oo$.
\end{itemize}
}\end{lem}
\begin{proof}{\ \\
\begin{itemize}
\item[{\rm (a)}] Let $(\oo,\ii_{_+},\ii_{_-})$ be a solution with $\ii_{_-}\neq 0$ and $\ii_{_+}\neq 0$, and define the
new parameters $\ii'_{_+}$ and $\ii'_{_-}$ as,
$$\left\lbrace\begin{array}{lll}
\ii'_{_+}=\ii_{_+}-\ii_{_-},\ \ii'_{_-}=0 &\qquad & {\rm if \ }\ii_{_+}-\ii_{_-}\geq 0,\\
\ii'_{_-}=\ii_{_-}-\ii_{_+},\ \ii'_{_+}=0 &\qquad & {\rm if \ }\ii_{_+}-\ii_{_-}< 0,
\end{array}\right.$$
and note that $(\oo,\ii'_{_+},\ii'_{_-})$ is a solution satisfying $\oo+\ii'_{_+}+\ii'_{_-}<\oo+\ii_{_+}+\ii_{_-}$,
which is a contradiction.
\item[{\rm (b)}] Suppose that $n$ is even and $k$ is odd. Since $\oo+k(\ii_{_+}-\ii_{_-})=tn$ is even, either both $\oo$
and $(\ii_{_+}-\ii_{_-})$ are even or both are odd. Thus,
$$\oo+\ii_{_+}+\ii_{_-}=(\oo+\ii_{_+}-\ii_{_-})+2\ii_{_-},$$
is even, which is a contradiction.
\item[{\rm (c)}] If $n$ is odd then $(\oo',\ii'_{_+},\ii'_{_-})=(n,0,0)$ is feasible. Also, if $n$ is even, then by part $(b)$, $k$ is even, and one may verify that $(\oo',\ii'_{_+},\ii'_{_-})=(k,0,1)$ is feasible. In both cases, since we have an objective less than or equal to $n$ and $(\oo,\ii_{_+},\ii_{_-})$ is a minimizer, we have $ \oo+\ii_{_-}+\ii_{_+}\leq n$.
\item[{\rm (d)}] By contradiction assume that $\oo\geq k$. Using $(a)$ we may assume $\ii_{_-}= 0$ or $\ii_{_+}= 0$ and define
$$\left\lbrace\begin{array}{llll}
\oo'=\oo-k,\quad &\ii'_{_-}=\ii_{_-}=0,\quad & \ii'_{_+}=\ii_{_+}-1, \quad &{\rm if}\ \ii_{_+}> 0,\\
\oo'=\oo-k,\quad & \ii'_{_-}=\ii_{_-}+1,&\ii'_{_+}=0, &{\rm if}\ \ii_{_+}= 0.
\end{array}\right.$$
Now note that $(\oo',\ii'_{_+},\ii'_{_-})$ is feasible for $(*)$, satisfying
$$\oo'+\ii'_{_+}+\ii'_{_-}\leq \oo+\ii_{_+}+\ii_{_-}-(k-1)<\oo+\ii_{_+}+\ii_{_-},$$
which is a contradiction.

 \item[{\rm (e)}] If $k$ is even, then $(\oo',\ii'_{_+},\ii'_{_-})=(k,0,1)$ is feasible.
Now, if $(\oo,\ii_{_+},\ii_{_-})$ is a solution giving rise to an odd objective less than $k+3$ and $\oo> k$, then we should have $k<\oo\leq \oo+(\ii_{_+}+\ii_{_-})\leq k+1$. Consequently, $\oo=k+1$, and we may conclude that $\ii_{_+}+\ii_{_-}=0$ and $k+1=\oo=\oo+k(\ii_{_+}-\ii_{_-})=tn>2k\geq k+2$, which is a contradiction.

 \item[{\rm (f)}] By the definition of $(*)$, we have
$\oo=tn-k(\ii_{_+}-\ii_{_-}).$
Since $\gcd(n,k)$ divides the right hand side, it divides $\oo$ too.
\end{itemize}
}\end{proof}

 The following simple observation as a corollary of Lemma~\ref{lem:basic}(a) is sometimes quite useful.
Note that by Lemma~\ref{lem:basic}($a$) one may define $\ii \isdef \ii_{_+}-\ii_{_-}$ and talk about
a $(\oo,\ii,r,t)$ solution of the system $(*)$. Clearly, for such a solution if $\ii \geq 0$ then one has
$(\oo,\ii_{_+}=\ii,\ii_{_-}=0,r,t)$ as a solution and if $\ii \leq 0$ then one has
$(\oo,\ii_{_+}=0,\ii_{_-}=-\ii,r,t)$ as a solution. Hereafter, we may freely talk about $(\oo,\ii,r,t)$ as a solution, adapting
this convention. Also, note that by Lemma~\ref{lem:cyc-odd}($a$) for any solution $(\oo, \ii_{_+},\ii_{_-})$ we have,
$$|\ii|=|\ii_{_+}-\ii_{_-}|=\ii_{_+}+\ii_{_-}.$$

 On the other hand, the set of solutions of $(*)$ is equal to the set of solutions of the following minimization problem,
$$
\begin{array}{llll}
\min & &\oo+ |\ii| &\\
& &\oo+|\ii| &=2r+1\\
& & \oo + k\ii & =tn \\
& &\oo& \geq 0.
\end{array}
$$
Now, by substitution for the variable $\oo$ we find the following minimization problem
whose set of solutions is equal to the set of solutions of $(*)$.
$$
\begin{array}{llll}
\min & & tn -k\ii + |\ii| &\\
& &tn -k\ii+|\ii| &=2r+1 \\
& &tn -k\ii& \geq 0.
\end{array} \qquad (**)
$$

 The following result is our first step toward a clarification of relationships between the set of solutions of $(*)$
and the odd girt of $\matr{Pet}(n,k)$.

 \begin{thm}{\label{thm:main} The following statements are equivalent,
\begin{itemize}
\item[{\rm (a)}] The feasible set of $(*)$ is non-empty.
\item[{\rm (b)}] $\matr{Pet}(n,k)$ has an odd cycle.
\end{itemize}
Also, if $(\oo, \ii_{_+}, \ii_{_-},r,t)$ is a solution of $(*)$, then the odd girth of $\matr{Pet}(n,k)$ is equal to $2r+1$
if $(\oo, \ii_{_+}, \ii_{_-},r,t)$ is a trivial solution $($i.e. either $\oo=0$ or $\ii_{_+}+\ii_{_-}=0)$. The odd girth
of $\matr{Pet}(n,k)$ is equal to $2r+3$, if all solutions of $(*)$ are nontrivial.
}\end{thm}
\begin{figure}[ht]
\label{system}
\centering{\includegraphics[width=12cm]{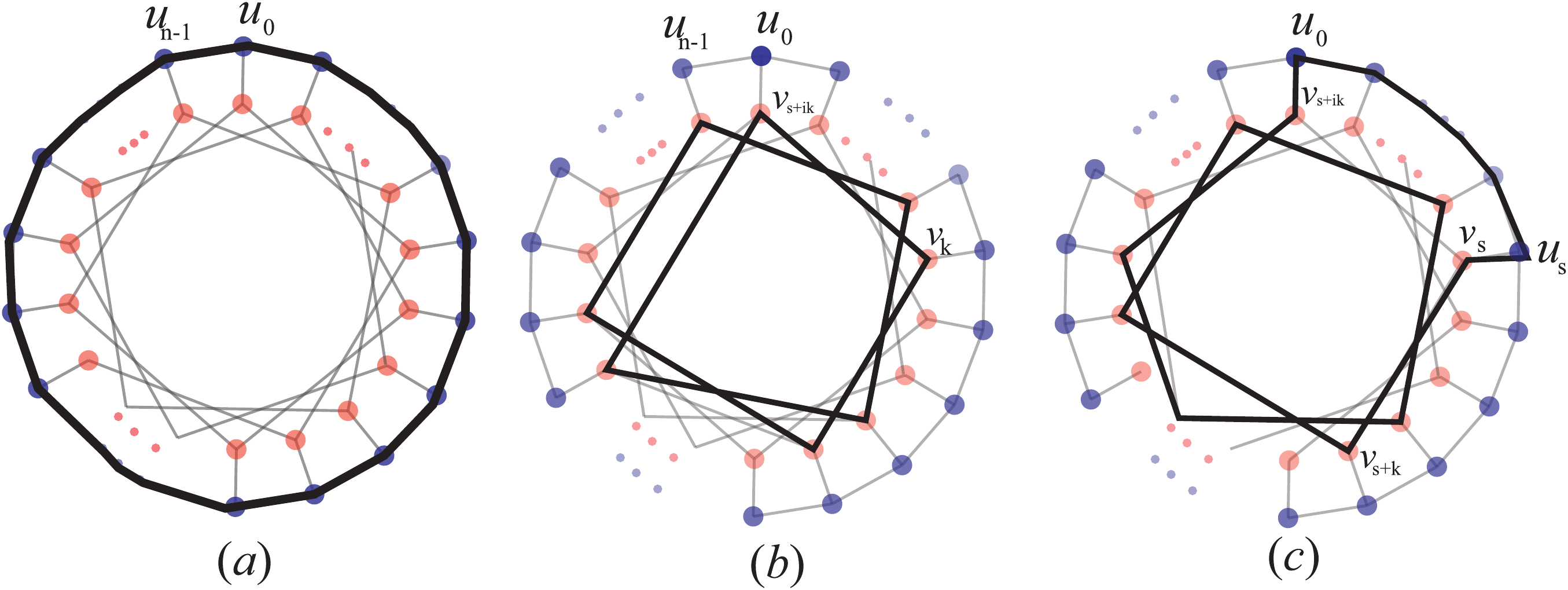}}
\caption{Constructing an odd cycle given $\ii_{_+},\ii_{_-}$ and $\oo$ with $\oo+k (\ii_{_+}-\ii_{_-})\stackrel{n}{\equiv} 0$ in three cases: $(a)$ $\ii_{_+}+\ii_{_-}=0$, \ $(b)$ $\oo=0$, $(c)$ $\ii_{_+}+\ii_{_-}\neq 0$ and $\oo\neq 0$.}
\end{figure}

 \begin{proof}{\ \\
\begin{itemize}
\item{$(a) \Rightarrow (b):$
Let $(\oo, \ii_{_+},\ii_{_-},r,t)$ be a feasible point of the system
$(*)$. First, we show that there exists an odd cycle $\matr{C}$ in $\matr{Pet}(n,k)$ of length
$\ell$, where
$$\ell \leq \left\lbrace\begin{array}{lll}
2r+1, &\qquad& {\rm if\ } \oo=0 \ {\rm or} \ \ii_{_+}+\ii_{_-}=0,\\
2r+3&\qquad& {\rm otherwise}.
\end{array}\right.$$
Consider the following three cases.
\begin{itemize}
\item[{\rm i)}]
If $\ii_{_+}+\ii_{_-}=0$, then we have $\oo \stackrel{n}{\equiv} 0$, $\oo=2r+1\neq 0$ and $0\leq \oo\leq n$. Consequently, $2r+1=\oo=n$, and
we have an odd cycle (the outer cycle $u_{_0}u_{_1}\dots u_{_{n-1}}u_{_0}$) of length $2r+1=n$ (see Figure~\ref{system}$(a)$).

 \item[{\rm ii)}] If $\oo=0$, then without loss of generality suppose that $\ii_{_-}\leq\ii_{_+}$, now
$$\ii\isdef \ii_{_+}-\ii_{_-}=\ii_{_+}+\ii_{_-}-2\ii_{_-}=2r+1-2\ii_{_-}=2r'+1> 0 \ \ {\rm and\ \ } k(\ii_{_+}-\ii_{_-}) \stackrel{n}{\equiv} 0.$$
Consequently, one can consider the odd closed walk $v_{_0}v_{_k}v_{_{2k}}\dots v_{_{(\ii-1)k}}v_{_0}$. The length of this closed walk is equal to $\ii=2r'+1$, and contains
at least one odd cycle of length less than or equal to $2r+1$ (see Figure~\ref{system}$(b)$).

 \item[{\rm iii)}]
If $\ii_{_+}+\ii_{_-} \neq 0$ and $\oo \neq 0$, consider the closed walk
$$\left\lbrace\begin{array}{ll}
u_{_{0}}u_{_{1}}\dots u_{_{\oo}}v_{_{\oo}}v_{_{\oo+k}}v_{_{\oo+2k}}\dots v_{_{\oo+(\ii_{_+}-\ii_{_-}-1)k}}v_{_{0}}u_{_{0}},\quad & {\rm if} \ \ii_{_+}\geq\ii_{_-},\\
u_{_{0}}u_{_{1}}\dots u_{_{\oo}}v_{_{\oo}}v_{_{\oo-k}}v_{_{\oo-2k}}\dots v_{_{\oo-(\ii_{_-}-\ii_{_+}-1)k}}v_{_{0}}u_{_{0}},\quad &{\rm if} \ \ii_{_+}< \ii_{_-}.
\end{array}\right.
$$
Note that in either case we have a closed walk of length $\oo+k|\ii_{_+}-\ii_{_-}|+2\leq 2r+3$ that contains an odd cycle of length less than or equal to $2r+3$ (see Figure~\ref{system}$(c)$).
\end{itemize}
}
\item{$(b) \Rightarrow (a):$
Let $\matr{C}_{_\mu}$ be a minimum odd cycle in $\matr{Pet}(n,k)$. Without loss of generality, one may assume that
$\oo_{_+}(\matr{C}_{_\mu})\geq \oo_{_-}(\matr{C}_{_\mu})$ (reverse the direction, otherwise). Also, let
$$\oo \isdef \oo_{_+}(\matr{C}_{_\mu})-\oo_{_-}(\matr{C}_{_\mu}), \ \ii_{_+} \isdef \ii_{_+}(\matr{C}_{_\mu}), {\rm and\ }
\ii_{_-}= \ii_{_-}(\matr{C}_{_\mu}).$$
Note that since $l(\matr{C}_{_\mu})=\oo+\ii_{_+}+\ii_{_-}+2\oo_{_-}(\matr{C}_{_\mu})+b(\matr{C}_{_\mu})$ is odd, and $b(\matr{C}_{_\mu})$ is even, then $\oo+\ii_{_+}+\ii_{_-}$ is odd. Also, by Lemma~\ref{lem:cyc-odd} we have
$\oo+k(\ii_{_+}-\ii_{_-})\stackrel{n}{\equiv} 0$, and consequently, $(\oo, \ii_{_+},\ii_{_-})$ is feasible for $(*)$.
}
\end{itemize}
Again, let $\matr{C}_{_\mu}$ be a minimum cycle with parameters $(\oo, \ii_{_+},\ii_{_-},r)$. To determine the length of
$\matr{C}_{_\mu}$ we consider the following two cases.
\begin{enumerate}
\item
{\it There is a trivial solution $(\oo^{*}, \ii_{_+}^{*},\ii_{_-}^{*},r^{*})$}.

 By $(a) \Rightarrow (b)$, there exists an odd cycle of length less than or equal to $2r^{*}+1$ and consequently,
$l(\matr{C}_{_\mu})\leq 2r^{*}+1$.

 Also, by $(b) \Rightarrow (a)$, parameters $(\oo, \ii_{_+},\ii_{_-},r)$ are feasible for $(*)$ and
$$l(\matr{C}_{_\mu})\geq \oo+\ii_{_+}+\ii_{_-}=2r+1\geq 2r^{*}+1.$$
Hence, $l(\matr{C}_{_\mu})\geq 2r^{*}+1$ which shows that $l(\matr{C}_{_\mu})= 2r^{*}+1$ in this case.
\item
{\it There is not any trivial solution and $(\oo^{*}, \ii_{_+}^{*},\ii_{_-}^{*},r^{*})$ is a non-trivial solution.}

 Similarly, by $(a) \Rightarrow (b)$, we have $l(\matr{C}_{_\mu})\leq 2r^{*}+3$. Also, using $(b) \Rightarrow (a)$, parameters
$(\oo, \ii_{_+},\ii_{_-},r)$ are feasible for $(*)$ and consequently,
$$l(\matr{C}_{_\mu})\geq \oo+\ii_{_+}+\ii_{_-}+b(\matr{C}_{_\mu})=2r+1+b(\matr{C}_{_\mu})\geq 2r^{*}+1+b(\matr{C}_{_\mu}).$$
Note that if $l(\matr{C}_{_\mu}) < 2r^{*}+3$, then since $l(\matr{C}_{_\mu})$ is odd,
by the above inequality we have
$$2r^{*}+1 \geq l(\matr{C}_{_\mu})\geq 2r^{*}+1+b(\matr{C}_{_\mu}),$$
which implies that $b(\matr{C}_{_\mu})=0$, and consequently by Lemma~\ref{lem:cyc-odd}(e) either $\oo=0$
or $\ii_{_+}+\ii_{_-}=0$. But since $(\oo,\ii_{_+},\ii_{_-},r)$ is feasible and $2r+1\leq 2r^*+1$
one concludes that $(\oo,\ii_{_+},\ii_{_-})$ is a \textit{trivial solution} which is a contradiction. Therefore,
$l(\matr{C}_{_\mu})\geq 2r^{*}+3$ and consequently, $l(\matr{C}_{_\mu})= 2r^{*}+3$. (As a byproduct we have also proved
that $b(\matr{C}_{_\mu}) \geq 2$ in this case.)
\end{enumerate}
}\end{proof}

 A couple of remarks on the proof of the previous theorem are instructive.
First, note that by the argument appearing in part (i) of the $(a) \Rightarrow (b)$ section of the proof we may conclude that
for any trivial solution $(\oo, \ii_{_+},\ii_{_-})$ if $\ii_{_+}+\ii_{_-}=0$ then $\oo$ is odd and equal to $n$. This shows that
if $n$ is even then we have $\ii_{_+}+\ii_{_-} \neq 0$.

 On the other hand, if $\oo=0$ then one may note that $(\oo=0,\ii_{_-}=0,\ii_{_+}=\frac{n}{\gcd(n,k)})$ is a solution. This not only shows that $\frac{n}{\gcd(n,k)}$ being an even number implies that $\oo\neq 0$, but also the above argument shows that
if $\oo+\ii_{_+}+\ii_- \notin \{n, \frac{n}{\gcd(n,k)}\}$ then $(\oo, \ii_{_+},\ii_{_-})$ is a nontrivial solution.

 Also, note that by the proof of the last part of the theorem, if $(\oo, \ii_{_+}, \ii_{_-},r,t)$ is a solution of $(*)$
and the odd girth of $\matr{Pet}(n,k)$ is equal to $2r+3$, then any cycle of minimum length will intersect both outer and inner cycles of $\matr{Pet}(n,k)$.

 Moreover, one may reprove the following classic result as a consequence of our discussions so far.

 \begin{cor}\label{cor:bipartite}{The generalized Petersen graph, $\matr{Pet}(n,k)$, is bipartite if and only if $n$ is even and $k$ is odd.
}\end{cor}
\begin{proof}{First, note that if $n$ is even and $k$ is odd, then by Lemma~\ref{lem:basic}$(b)$ the feasible set of
the system $(*)$ is empty. Hence, by Theorem~\ref{thm:main} $\matr{Pet}(n,k)$ is bipartite since it does not have any odd cycle.

 On the other hand, if $n$ is odd, then $\oo=n$ and $\ii_{_-}=\ii_{_+}=0$ constitute a feasible set of parameters for
the system $(*)$, implying the existence of an odd cycle by Theorem~\ref{thm:main}.

 Also, if both $n$ and $k$ are even, then similarly, $\oo=k$ and $\ii_{_+}=0, \ii_{_-}=1$,
constitute a feasible set of parameters for
the system $(*)$, again implying the existence of an odd cycle by Theorem~\ref{thm:main}.
}\end{proof}

 Now, we set for a proof of Theorem~\ref{thm:oddgirth}, but first we will be needing the following simple lemma.

 \begin{lem}{\label{lem:uniq}
Consider the equation
$\oo+k \ii=tn,$
in which $k \in \mathbb{N}$, $n \in \mathbb{N}$ and $t \in \mathbb{Z}$ are constants and $\oo \in \mathbb{N}$ and
$\ii \in \mathbb{Z}$ are unknowns. Then
\begin{itemize}
\item[{\rm i)}]{If $t=0$ and $0 \leq \oo \leq k$} then the only solution
of this equation other than the trivial solution $(\oo,\ii)=(0,0)$ is $(\oo,\ii)=(k,-1)$.
\item[{\rm ii)}]{If $t \neq 0$ and $0 \leq \oo < k$ then the solution is uniquely determined as follows,
$$\left\lbrace
\begin{array}{lllll}
\oo=tn-k\lfloor \frac{tn}{k} \rfloor,&\quad& \ii=\lfloor \frac{tn}{k} \rfloor,&\quad& t>0,\\
\oo=k\lceil \frac{-tn}{k}\rceil +tn, &\quad& \ii=-\lceil \frac{-tn}{k}\rceil,&\quad& t<0.
\end{array}\right.
$$
}
\end{itemize}
}\end{lem}

 \begin{proof}{
Part (i) is clear. For part (ii) one may easily verify that the provided expressions satisfy the equation and its conditions.
On the other hand, let's assume that we have two solutions, namely,
$$\oo+k \ii=tn, \quad {\rm and} \quad \oo'+k\ii'=tn. $$
Without loss of generality, we may assume $\oo\geq \oo'$, and consequently,
$(\oo-\oo')+k(\ii-\ii')=0.$
But by part (i) and the fact that $0 \leq \oo-\oo' < k$ we have $\oo-\oo'=\ii-\ii'=0$, proving the uniqueness.
}\end{proof}

 This shows that if we have a nontrivial solution, and $\oo \neq k$ then either
$\oo=tn-k\lfloor \frac{tn}{k} \rfloor, \ii=\lfloor \frac{tn}{k} \rfloor$ or $\oo=k\lceil \frac{tn}{k}\rceil -tn, \ii=-\lceil \frac{tn}{k}\rceil$ for some positive integer $t$.

 \begin{proof}{(of Theorem~\ref{thm:oddgirth})

 Let $\matr{Pet}(n,k)$ be a non-bipartite graph.
We consider the following two cases.
\begin{itemize}
\item{{\it There exist a trivial solution $(\oo,\ii,r,t)$ for $(*)$}:\\
We show that in this case the odd girth is equal to $\frac{n}{\gcd(n,k)}$. Note that if $(\oo,\ii)=(0,\frac{n}{\gcd(n,k)})$,
then by the discussion proceeding Theorem~\ref{thm:main} the odd girth is equal to $\frac{n}{\gcd(n,k)}$.

 On the other hand, if $\ii=0$, then $\oo=n$ is odd and equal to the odd girth by Theorem~\ref{thm:main}.
Also, note that in this case $(\oo=0,\ii_{_+}=\frac{n}{\gcd(n,k)},\ii_{_-}=0)$ is feasible for $(*)$ and consequently,
$n \leq \frac{n}{\gcd(n,k)}$. This implies $\gcd(n,k)=1$ showing that the odd girth is equal to $n=\frac{n}{\gcd(n,k)}$.
}
\item{{\it There is not any trivial solution and $(\oo,\ii,r,t)$ is a solution of $(*)$}:\\
First, we prove the following claims. Recall that $par(k)$ is the parity function which is equal to one when $k$ is odd and is zero otherwise.
\begin{itemize}
\item[{\rm i)}]{There exists a nontrivial solution $(\oo,\ii,r,t)$ for which $|t| \leq \lfloor\frac{(k-1)^2}{n}\rfloor+par(k)$.\\
We consider the following cases.
\begin{itemize}
\item[-]{{\it $k$ and $n$ are odd.}\\
Note that $(\oo'=n-k\lfloor \frac{n}{k}\rfloor,\ii_{_+}'=\lfloor \frac{n}{k}\rfloor,\ii_{_-}'=0)$ is feasible for $(*)$,
and consequently, for any nontrivial solution $(\oo,\ii,r,t)$ we have
$$\oo+|\ii|\leq n-\lfloor \frac{n}{k}\rfloor(k-1)< n- (\frac{n}{k}-1)(k-1)=\frac{n}{k}+k-1.$$
Hence,
$$|t|n=|\oo+k \ii|\leq \oo+|\ii|k < \oo+(\frac{n}{k}+k-1-\oo)k=\oo(1-k)+n+k^2-k\leq n+(k-1)^2,$$
implying that
$$|t| \leq \lfloor\frac{(k-1)^2}{n}\rfloor+1.$$
}
\item[-]{{\it $k$ is even.}\\
Note that $(\oo'=k,\ii_{_+}'=0,\ii_{_-}'=1,t'=0)$ is feasible for $(*)$.
If this is a solution then the inequality is trivially satisfied. Otherwise, for any nontrivial solution $(\oo,\ii,r,t)$
we have $\oo+|\ii|< k+1$, implying that $\oo+|\ii| \leq k-1$ since $k$ is even. On the other hand, $\oo\geq 1$ implies that $|\ii|\leq k-2$, and consequently,
applying Lemma~\ref{lem:uniq} we have,
$$|t|n=|\oo+k \ii|\leq \oo+|\ii|k \leq (\oo+|\ii|)+(k-1)|\ii|\leq k-1+(k-1)(k-2)=(k-1)^2.$$
This shows that $$|t| \leq \lfloor\frac{(k-1)^2}{n}\rfloor.$$
}
\end{itemize}
}
\item[{\rm ii)}]{{\it There exists a nontrivial solution $(\oo,\ii,r,t)$ for which $|t|\leq \frac{2k}{\gcd(n,k)}$.
Moreover, if $\frac{n}{\gcd(n,k)}$ is even, then there exists a nontrivial solution $(\oo,\ii,r,t)$ for which \linebreak
$|t|\leq \frac{k}{\gcd(n,k)}$.}\\
First, assume that for positive integers $\alpha$ and $\beta < t$ we have $t=\beta +\frac{\alpha k}{\gcd(n,k)}$.
Then one may verify that,
$$\begin{array}{ll}
tn+(1-k)\lfloor \frac{tn}{k} \rfloor&=
(\frac{\alpha k}{\gcd(n,k)}+\beta )n+(1-k)\lfloor \frac{(\frac{\alpha k}{\gcd(n,k)}+\beta )n}{k}\rfloor\\
&= \left( \beta n+(1-k)\lfloor \frac{\beta n}{k}\rfloor\right)+\left( (\frac{\alpha k}{\gcd(n,k)})n+(1-k)(\frac{\alpha n}{\gcd(n,k)}) \right)
\quad\\
&= \left( \beta n+(1-k)\lfloor \frac{\beta n}{k}\rfloor\right)+\left( \frac{\alpha n}{\gcd(n,k)} \right)
\quad
\end{array}$$
$$\begin{array}{ll}
(1+k)\lceil \frac{tn}{k}\rceil -tn&=
(1+k)\lceil \frac{(\frac{\alpha k}{\gcd(n,k)}+\beta )n}{k}\rceil -(\frac{\alpha k}{\gcd(n,k)}+\beta )n\\
&= \left((1+k)\lceil \frac{\beta n}{k}\rceil-\beta n \right)+\left((1+k)(\frac{\alpha n}{\gcd(n,k)})-(\frac{\alpha k}{\gcd(n,k)})n \right)
\quad\\
&= \left( (1+k)\lceil \frac{\beta n}{k}\rceil- \beta n \right)+\left( \frac{\alpha n}{\gcd(n,k)} \right)
\quad
\end{array}$$
To prove the claim we proceed by contradiction. First, note that by parts $(d)$ and $(e)$ of Lemma~\ref{lem:basic}, either
$(\oo,\ii,t)=(k,-1,0)$ or $0 \leq \oo < k$. Clearly, the claims are true for $t=0$, hence we may assume that $t \neq 0$
and $0 \leq \oo < k$. On the other hand by Lemma~\ref{lem:uniq} we know that $\oo+\ii$ is equal to
$$\left\lbrace
\begin{array}{ll}
\oo+\ii=tn+(1-k)\lfloor \frac{tn}{k} \rfloor,& t>0,\\
&\\
\oo+\ii=(1+k)\lceil \frac{-tn}{k}\rceil +tn, & t<0
\end{array}\right.
$$
as the minimum value of the solutions of $(*)$. But by setting
$\alpha =1$ when $\frac{n}{\gcd(n,k)}$ is even and $\alpha =2$ when $\frac{n}{\gcd(n,k)}$ is odd,
and the above computation we see that one finds the minimum values
$$\left( \beta n+(1-k)\lfloor \frac{\beta n}{k}\rfloor \right)$$
or
$$\left( (1+k)\lceil \frac{\beta n}{k}\rceil - \beta n \right)$$
for the solutions of $(*)$ which are strictly smaller, and this is a contradiction.
}
\end{itemize}
}
Now, the rest of the proof is clear by the above case studies since the set $\matr{Ind}(n,k)$ essentially categorizes different cases and the minimization provides that odd girth by Theorem~\ref{thm:main}.
\end{itemize}
}\end{proof}
Note that Theorem~\ref{thm:oddgirth} not only provides an explicit expression for the odd girth of $\matr{Pet}(n,k)$,
but also the provided expression is also effectively computable in many cases. The following corollary and example will
show some special cases.
\begin{cor}{For any odd number $n$, there exists two generalized Petersen graphs
$\matr{Pet}(2n+1,2d+1)$ and $\matr{Pet}(2n+1,2d)$ with odd girth $n$.
}\end{cor}
\begin{exm}{\label{exm:(n,2)}
Let us consider the following special cases as a couple of concrete examples.
\begin{itemize}
\item {\it The odd girth of $\matr{Pet}(n,2)$ is equal to $5$, except for the case $n=6$, for which odd girth is
equal to $3$.}\\
The case $n=5$ is well-known. If $n > 5$ then $n>(k-1)^2$, and consequently, for $\frac{n}{(n,2)}\geq 5$, we have $|t|\leq \lfloor \frac{(k-1)^2}{n} \rfloor=0$.
Hence, in this case the odd girth is equal to $k+3=5$. The only case for which we have $\frac{n}{(n,2)}<5$ is $n=6$, in which ase the odd girth is equal to $\frac{6}{(6,2)}=3$.
\item
{\it For $\matr{Pet}(n,3)$, we have}
$$
\go(\matr{Pet}(n,3))=\left\lbrace
\begin{array}{lll}
\frac{n}{3}&\quad & 3|n,\\
\frac{n+8}{3}&\quad & n\stackrel{3}{\equiv}1,\\
\frac{n+10}{3}&\quad & n\stackrel{3}{\equiv}2.\\
\end{array}\right.$$
Note that for $n\geq 5$ we have $|t|\leq \lfloor \frac{(k-1)^2}{n} \rfloor+1=1$. Therefore, we just need to consider the special cases $t=1$, $t=-1$ and the trivial solutions.
\end{itemize}
}\end{exm}

 The following result is essentially a direct consequence of Theorem~\ref{thm:oddgirth}, but we provide some direct proofs for clarity and simplicity.
\begin{proof}{(of Theorem~\ref{thm:girthbounds})\\
For the upper bounds note that
\begin{itemize}
\item{If $k$ is even and $\go(\matr{Pet}(n,k))\neq k+3$ then by feasibility of $(\oo,\ii_{_+},\ii_{_-})=(k,0,1)$ we have
$\go(\matr{Pet}(n,k))\leq k+1$.
}
\item{If $k$ is odd then $n$ is also odd by Corollary~\ref{cor:bipartite}.
Also, considering a feasible point uniquely determined
by Lemma~\ref{lem:uniq} for $t=1$, by and Theorem~\ref{thm:main} we have
$$\go(\matr{Pet}(n,k))\leq n+(1-k)\lfloor \frac{n}{k} \rfloor+2\leq n+(1-k)(\frac{n}{k}-1)+2=\frac{n}{k}+k+1.$$
}
\end{itemize}
For the lower bounds, first, and by contradiction, let $(\oo,\ii_{_+},\ii_{_-},r,t)$ be a solution of $(*)$ such that $\oo+\ii_{_-}+\ii_{_+}< \frac{n}{k}$, then
$$\oo+k(\ii_{_+}-\ii_{_-})\leq \oo+k(\ii_{_+}+\ii_{_-})\leq k(\oo+\ii_{_+}+\ii_{_-})<k\frac{n}{k}=n,$$
$$\oo+k(\ii_{_+}-\ii_{_-})\geq \oo-k(\ii_{_+}+\ii_{_-})\geq -k(\oo+\ii_{_+}-\ii_{_-})>-k\frac{n}{k}=-n.$$
So $-n<\oo+k(\ii_{_+}-\ii_{_-})=tn<n$, thus we have $\oo+k(\ii_{_+}-\ii_{_-})=t=0$.

 Now, if $k$ is even this contradicts Lemma~\ref{lem:basic}(e). On the other hand, if
$k$ and $n$ are odd, since $\oo+k(\ii_{_+}-\ii_{_-})=0$, either both $\oo$ and $\ii_{_+}-\ii_{_-}$ are odd or both of them are even. Hence, $\oo+\ii_{_+}+\ii_{_-}=(\oo+\ii_{_+}-\ii_{_-})+2\ii_{_-}$ is even, which is contradiction, since
we already know that $\oo+\ii_{_+}+\ii_{_-}=2r+1$.

 For the second term in the lower bound, consider a solution of $(**)$ as $tn -k\ii + |\ii| \geq 1$. Then one may verify that,
$$ tn -k\ii + |\ii| =
\left\lbrace\begin{array}{ll}
tn-k\ii+\ii=tn-(k-1)\ii &\ii> 0\\
tn-k\ii-\ii=tn-(k+1)\ii \ & \ii< 0\\
tn & \ii=0
\end{array}\right.
\geq \left\lbrace\begin{array}{lll}
\gcd(n,k-1) \ &\ii> 0\\
\gcd(n,k+1) & \ii< 0\\
k+1 \geq \gcd(n,k-1)+2& \ii=0.
\end{array}\right.$$

 Therefore, if $\ii=0$ then the solution is greater than or equal to $\gcd(n,k-1)+2$ (note that by
Definition~\ref{def:Petersen} we know that $2< 2k\leq n$). Otherwise,
the odd girth is greater than or equal to $\min\{\gcd(n,k-1),\gcd(n,k+1)\}$.

 }\end{proof}

 \section{On circular chromatic number of $\matr{Pet}(n,k)$}\label{sec:circhrom}

 In this section we will prove Theorems~\ref{thm:homk2ko} and \ref{thm:uphom}. To start, let us recall the following results
from \cite{haji10} and \cite{haji11}.
\begin{alphthm}{{\rm\cite{haji10}}\label{lem:haji}
Let $\matr{G}$ and $\matr{H}$ are two graphs, where $2q+1< \go(\matr{H})$, then,
$$\matr{G}^{^{\frac{1}{2q+1}}}\rightarrow \matr{H} \quad \Longleftrightarrow \quad \matr{G}\rightarrow \matr{H}^{^{2q+1}}.$$
}\end{alphthm}

 \begin{alphthm}{{\rm \cite{haji11}}\label{thm:haji} Let $\matr{G}$ be a non-bipartite graph with circular chromatic number $\chi_{_c}(\matr{G})$. Then for any positive integer $s$, we have
$$\chi_{_c}(\matr{G}^{^{\frac{1}{2s+1}}})=\frac{(2s+1)\chi_{_c}(\matr{G})}{s\chi_{_c}(\matr{G})+1}.$$
}\end{alphthm}

 Also, we will need the following lemma. Note that in what follows we use the crucial observation that for an odd power graph
$\matr{G}^{^{2q+1}}$, two vertices $u$ and $v$ are adjacent if and only if there exists a path of odd length less than or equal to $2q+1$ between them in $\matr{G}$. This in particular shows that if $p \leq q$ then $\matr{G}^{^{2p+1}}$ is a subgraph of
$\matr{G}^{^{2q+1}}$.
\begin{pro}{\label{pro:power}
Let $\matr{C} \isdef w_{_0}w_{_1}w_{_2}\dots w_{_{s-1}}w_{_0}$ be an odd cycle of the graph $\matr{G}$, where
$u,v\in V(\matr{G})\setminus \{w_{_0}, w_{_1},w_{_2}, \dots , w_{_{s-1}}\}$.
Also, let $u$ be adjacent to $w_{_i}$ and $v$ be adjacent to $w_{_{j}}$ in $\matr{G}$ for some $0\leq i < j\leq s-1$. Then in the power graph $\matr{G}^{^{s-2}}$,
\begin{itemize}
\item[$(a)$] $w_{_{i}}$ and $w_{_{j}}$ are adjacent,
\item[$(b)$] If $j-i\stackrel{s}{\not \equiv}\pm 1$, then $u$ is adjacent to $w_{_j}$,
\item[$(c)$] If $j-i\stackrel{s}{\not \equiv}\pm 2$, then $u$ is adjacent to $v$.
\end{itemize}
}\end{pro}
\begin{proof}{Since $s-2$ is odd, in each case we show that there exists an odd path of length less than or equal $s-2$ between the corresponding vertices in $\matr{G}$.
\begin{itemize}
\item[$(a)$] If $j-i$ is odd, then we have an odd walk
$w_{_{i}}w_{_{i+1}}\dots w_{_{j}}$ with length less than or equal $s-2$. If $j-i$ is even consider the
odd walk $w_{_{j}}w_{_{j+1}}\dots w_{_{s-1}}w_{_{0}}w_{_{1}}\dots w_{_{i}}$.
\item[$(b)$] If $j-i$ is odd, then we have an odd walk $w_{_{j}}w_{_{j+1}}\dots w_{_{s-1}}w_{_{0}}w_{_{1}}\dots w_{_{i}}u$,
of length $s-j+i+1$ less than or equal $s-2$. If $j-i$ is even consider the odd walk $uw_{_{i}}w_{_{i+1}}\dots w_{_{j}}$.
\item[$(c)$] If $j-i$ is odd, then we have an odd walk $uw_{_{i}}w_{_{i+1}}\dots w_{_{j}}v$ of length $j-i+2$ less than or equal $s-2$.
%(since $j-i\stackrel{s}{\neq}\pm 2$, we have $j-i<s-2$ and because $j-i$ is odd $j-i\leq s-4$).
If $j-i$ is even consider the odd walk
$vw_{_{j}}w_{_{j+1}}\dots w_{_{s-1}}w_{_{0}}w_{_{1}}\dots w_{_{i}}u$ of length $s-j+i+2$ less than or equal $s-2$.
% (since $j-i\stackrel{s}{\neq}\pm 2$, we have $j-i>2$ and because $j-i$ is even $j-i\geq 4$).
\end{itemize}
}\end{proof}
\begin{proof}{(of Theorem \ref{thm:homk2ko}$(a)$)\\
We prove the result for $\oo= tn-k\lfloor \frac{tn}{k} \rfloor$,
$\ii=\lfloor \frac{tn}{k} \rfloor$ and $2r+1=\ii+\oo$ (i.e. $t > 0$).\\
(for $\oo= k\lceil \frac{-tn}{k}\rceil +tn, \ii=-\lceil \frac{-tn}{k}\rceil$ (i.e. when $t < 0$)
one may use a similar argument.)

 Consider the set
$$\matr{P}=\{u_{_{0}},\cdots,u_{_{\oo}}\} \cup \{v_{_{0}},\cdots,v_{_{\oo}}\}
\cup \{u_{_{\oo+k}},\cdots,u_{_{\oo+(\ii-1)k}}\} \cup \{v_{_{\oo+k}},\cdots,v_{_{\oo+(\ii-1)k}}\},
$$
as a subset of $V(\matr{Pet}(n,k))$ and note that $|\matr{P}|=4r+2$.
Now, it is enough to show that $\matr{P}$ constitutes a clique in $\matr{Pet}(n,k)^{^{2r+1}}$.
Clearly, for this, it is enough to show that there exists an odd path of length less than or equal to
$2r+1$ between any pair of vertices in $\matr{P}$.

 For each $0 \leq i \leq n-1$ consider the $(2r+3)$-cycle, $\matr{C}_{_{i}}$, defined as follows
$$\matr{C}_{_{i}} \isdef u_{_{i}}u_{_{i+1}}\dots u_{_{i+\oo}}v_{_{i+\oo}}v_{_{i+\oo+k}}v_{_{i+\oo+2k}}\dots v_{_{i+\oo+(\ii-1)k}}v_{_{i}}u_{_{i}},$$
where summations are in modulo $n$.
Applying Proposition~\ref{pro:power}($a$) to $\matr{C}_{_{0}}$ we find out that
$$\{u_{_{0}},\cdots,u_{_{\oo}}\} \cup \{v_{_{\oo+k}},\cdots,v_{_{\oo+(\ii-1)k}}\} \cup \{v_{_{0}},v_{_{\oo}}\}$$
constitutes a clique in $\matr{Pet}(n,k)^{^{2r+1}}$.

 We claim that, if $i,j \in \{0,1,2, \dots, \oo\}$, and $h,\ell \in \{0,1,2,3,...,\ii-1\}$,
then there exists an odd path of length less than or equal to $2r+1$ between the following pairs of vertices
in $\matr{Pet}(n,k)$.
%1 and 4 are ok in C_0
% 1 and 1
%4 and 4

 %following cases
%1 and 2
%2 and 2
%%%%%Not 2 and 4
%3 and 4
%3 and 3
% 3and 2
%3 and 1
\begin{itemize}
\item{{\it $v_{_i}$ to $u_{_j}$.}\\ %1 and 2
Applying Proposition~\ref{pro:power}$(a)$ to $\matr{C}_{_{i}}$ the claim is true for $j \geq i$. For $j < i$
use the same argument on $\matr{C}_{_{i-\oo}}$.
}

 \item{{\it $v_{_i}$ to $v_{_j}$.}\\ %2 and 2
For $j > i$ note that $v_{_j}$ is connected to $u_{_j}$ in $\matr{C}_{_{i}}$ and the distance between
$v_{_i}$ and $u_{_j}$ in $\matr{C}_{_{i}}$ is greater than one. Hence,
Proposition~\ref{pro:power}$(b)$ proves this case. If $j < i$ then use the same reasoning on $\matr{C}_{_{i-\oo}}$.
}

 \item{{\it $u_{_{\oo+hk}}$ to $v_{_{\oo+\ell k}}$.}\\ %3 and 4
If $\ell \geq h$ invoke Proposition~\ref{pro:power}$(a)$ for $\matr{C}_{_{hk}}$. If $\ell < h$ consider
$\matr{C}_{_{(h-\ii)k}}$, note that $u_{_{\oo+hk}}$ is connected to $v_{_{\oo+hk}}$ in this cycle, and invoke
Proposition~\ref{pro:power}$(b)$.
}
\item{{\it $u_{_{\oo+hk}}$ to $u_{_{\oo+\ell k}}$.}\\ %3 and 3
Note that if $|\ell-h| \neq 2$, we may invoke
Proposition~\ref{pro:power}$(c)$ in $\matr{C}_{_{0}}$. Also, if $|\ell-h|=2$, we may
applying Proposition~\ref{pro:power}$(b)$ on $\matr{C}_{_{hk}}$ if $\ell > h$. If $|\ell-h|=2$ and $\ell < h$, consider $\matr{C}_{_{(h-\ii)k}}$ and invoke
Proposition~\ref{pro:power}$(b)$. \\
}

 \item{{\it $u_{_{\oo+hk}}$ to $v_{_i}$.}\\ % 2 and 3
If $h\neq 0$, consider $\matr{C}_{_{0}}$ and Proposition~\ref{pro:power}$(c)$, (note that in this case the distance
between $v_{_{\oo+hk}}$ and $u_{_i}$ is not equal to $2$). Also, the case $h=0$ is the case for $u_{_{i}}$ and $v_{_j}$ mentioned before.
}

 \item{{\it $u_{_{\oo+hk}}$ to $u_{_i}$.}\\ %1 and 3
If $h=0$ this is the case of $u_{_{i}}$ and $u_{_j}$ mentioned before. If $h\neq 0$,
the distance between $v_{_{\oo+hk}}$ and $u_{_i}$ in $\matr{C}_{_{0}}$ is not equal to $1$. Therefore,
we can apply Proposition~\ref{pro:power}$(b)$.
}

 \item{{\it $v_{_{\oo+hk}}$ to $v_{_i}$.}\\ %2 and 4
If $h=0$ this is the case of $v_{_{i}}$ and $v_{_j}$. Also, for $i=0$ we can apply
Proposition~\ref{pro:power}$(a)$ on $\matr{C}_{_{0}}$.
If $i\neq 0$ and $h\neq 0$, the distance between $v_{_{\oo+hk}}$ and $u_{_i}$ in $\matr{C}_{_{0}}$ is not equal to $1$. Hence,
we can apply Proposition~\ref{pro:power}$(b)$.
}
\end{itemize}
}\end{proof}

 \begin{cor}{
If the system $(*)$ does not have any trivial solution and the odd girth of $\matr{Pet}(n,k)$ is equal to $2r+3$, then
$$\matr{Pet}(n,k)\nrightarrow \matr{C}_{_{2r+3}}.$$
}\end{cor}
\begin{proof}{By contradiction assume that
$$\matr{Pet}(n,k)\rightarrow \matr{C}_{_{2r+3}}.$$
Then by Theorem~\ref{thm:homk2ko} and the functorial property of the power construction we have
$$ \matr{K}_{_{4r+2}} \rightarrow \matr{Pet}(n,k)^{^{2r+1}}\rightarrow \matr{C}_{_{2r+3}}^{^{2r+1}}=\matr{K}_{_{2r+3}},$$
which is a contradiction.
}\end{proof}

 \begin{proof}{(of Corollary~\ref{cor:compk}$(a)$)\\
Since $(2r+1) < \go (\matr{Pet}(n,k))=2r+3$, by Theorem~\ref{lem:haji} we have
$$\matr{K}_{_{4r+2}}\rightarrow \matr{Pet}(n, k)^{^{2r+1}} \Rightarrow \matr{K}_{_{4r+2}}^{^{\frac{1}{2r+1}}} \rightarrow \matr{Pet}(n, k).$$
Therefore, by Theorem~\ref{thm:haji} we have,
$$\chi_{_c}(\matr{Pet}(n, k)) \geq \chi_{_c}(\matr{K}_{_{4r+2}}^{^{\frac{1}{2r+1}}}) =\frac{(2r+1)\chi_{_c}(\matr{K}_{_{4r+2}})}{r\chi_{_c}(\matr{K}_{_{4r+2}})+1}= \frac{(2r+1)(4r+2)}{r(4r+2)+1}=2+ \dfrac{4r}{4r^2+2r+1}.$$
}\end{proof}
\begin{exm}{Consider the special case in which $n\leq 4k^2+1$, $n\stackrel{2k}{\equiv} 1$ and $\dfrac{n-1}{k}$ is even.
Then $(*)$ has no trivial solution and
$\go (\matr{Pet}(n,k))\leq \dfrac{n-1}{2k}+3$.
Therefore, by setting $2r+1=\dfrac{n-1}{2k}+1$, we have
$$\chi_{_c}(\matr{Pet}(n, k)) \geq 2+ \dfrac{4r}{4r^2+2r+1}.$$
In particular, if $n=4sk+1$, where $\oo\leq k$, by setting $r=s$, we have,
$$\chi_{_c}(\matr{Pet}(4sk+1, 2k)) \geq 2+ \dfrac{4s}{4s^2+2\oo+1}.$$
For example, for $\oo=1$, we have
$$\chi_{_c}(\matr{Pet}(4k+1, 2k)) \geq 2+ \dfrac{4}{7},$$
implying
$\matr{Pet}(4k+1, 2k) \nrightarrow \matr{C}_{5}.$
}\end{exm}

 \begin{exm}{
\begin{enumerate}
\item
Let $k=3$. For any odd $n$, where $3\nmid n$, by Theorem~\ref{thm:oddgirth}, we have
$$2r+1=\left\lbrace
\begin{array}{ll}
n-2\frac{n-1}{3}=\frac{n+2}{3} \quad &n\stackrel{3}{\equiv} 1,\\
n-2\frac{n-2}{3}=\frac{n+4}{3} \quad &n\stackrel{3}{\equiv} 2.\\
\end{array}\right.
$$
Therefore,
$$\chi_{_c}(\matr{Pet}(n, 3)) \geq
\left\lbrace
\begin{array}{ll}
2+\frac{6n-6}{n^2+n+7}
\quad &n\stackrel{3}{\equiv} 1,\\
2+\frac{6n+6}{n^2+5n+13}
\quad &n\stackrel{3}{\equiv} 2.\\
\end{array}\right.$$
For instance, $\chi_{_c}(\matr{Pet}(11, 3)) \geq 2+\dfrac{72}{189}$. Note that since
$$\chi_{_c}(\matr{C}_{_7})=2+\frac{1}{3}<2+\dfrac{72}{189}=\chi_{_c}(\matr{Pet}(11, 3))$$
one may conclude that $\matr{Pet}(11, 3)\nrightarrow \matr{C}_{_7}$.
Similarly,
$$\chi_{_c}(\matr{Pet}(7, 3)) \geq 2+\dfrac{36}{63}>2+\frac{1}{2}=\chi_{_c}(\matr{C}_{_5}),$$
and we have $\matr{Pet}(7, 3)\nrightarrow \matr{C}_{_5}$.
\item
If $n=21$ and $k=5$, by Theorem~\ref{thm:oddgirth} we have
$2r+1=5$, and consequently,
$$\chi_{_c}(\matr{Pet}(21, 5)) \geq 2+\dfrac{8}{21}> 2+\frac{1}{3},$$
implying $\matr{Pet}(21, 5)\nrightarrow \matr{C}_{_7}$.
\end{enumerate}
}\end{exm}

\begin{proof}{
(of Theorem~\ref{thm:homk2ko}$(b)$)\\
Reorganize the vertices of $\matr{Pet}(n,2k)^{^{2k+1}}$ as
$$x_{_0}=u_{_0}, x_{_1}=v_{_0},x_{_2}=u_{_1}, x_{_3}=v_{_1},\dots x_{_{2n-2}}= u_{_{n-1}}, x_{_{2n-1}}= v_{_{n-1}}.$$
In what follows we prove that each $x_{_{i}}$ is connected to $x_{_{i+\ell}}$ for $1 \leq \ell \leq 4k+1$ where summation is modulo $n$. This clearly implies that there is a subgraph in $\matr{Pet}(n,2k)^{^{2k+1}}$ isomorphic to $\overline{\matr{K}}_{_{\frac{2n}{4k+4}}}$.

 Note that, by symmetry, it is enough to prove the claim for $u_{_0}$ and $v_{_0}$. Hence, in what follows, we show that
for any $j \in\{1,2, \dots ,2k\}$ and in Watkins' notation, there exists an odd path of length less than or equal
to $2k+1$ between the following pairs of vertices in $\matr{Pet}(n,2k)$.
\begin{itemize}
\item[{\rm (1)}] $u_{_0}$ \textit{to} $u_{_j}$,\\
Apply Proposition~\ref{pro:power}$(a)$ in $\matr{C}_{_0}.$
\item[{\rm (2)}] $v_{_0}$ \textit{to} $u_{_j}$,\\
Apply Proposition~\ref{pro:power}$(a)$ in (see Figure~\ref{petcomp})
$$\matr{C}_{_0}=u_{_0}u_{_1}\ldots u_{_{2k}}v_{_{2k}}v_{_0}u_{_0}.$$
\item[{\rm (3)}] $v_{_0}$ \textit{to} $v_{_j}$,\\
If $j=2k$ we apply Proposition~\ref{pro:power}$(a)$ in $\matr{C}_{_0}$. Also,
if $j\neq 2k$, the distance between $v_{_0}$ and $u_{_j}$ is not equal to $1$ in $\matr{C}_{_0}$, and consequently, one may
invoke Proposition~\ref{pro:power}$(b)$.
\item[{\rm (4)}] $u_{_0}$ \textit{to} $v_{_j}$,\\
If $j=2k$ we apply Proposition~\ref{pro:power}$(a)$ in $\matr{C}_{_0}$. Also,
if $j\neq 2k$, and $j\neq 1$ the distance between $u_{_0}$ and $u_{_j}$ is not equal to $1$ in $\matr{C}_{_0}$, and consequently,
one may invoke Proposition~\ref{pro:power}$(b)$. On the other hand, if $j=1$ we apply Proposition~\ref{pro:power}$(a)$ in
(see Figure~\ref{petcomp})
$$\matr{C}_{_{n-2k-1}}=u_{_{n-2k-1}}u_{_{n-2k}}\ldots u_{_{0}}u_{_{1}}v_{_{1}}v_{_{n-2k-1}}v_{_{n-2k-1}}.$$
\item[{\rm (5)}] $v_{_0}$ \textit{to} $u_{_{2k+1}}$.\\
Consider the odd path $v_{_{0}}v_{_{2k}}u_{_{2k}}u_{_{2k+1}}$ (see Figure~\ref{petcomp}).
\end{itemize}
\begin{figure}[ht]
\centering{\includegraphics[width=12cm]{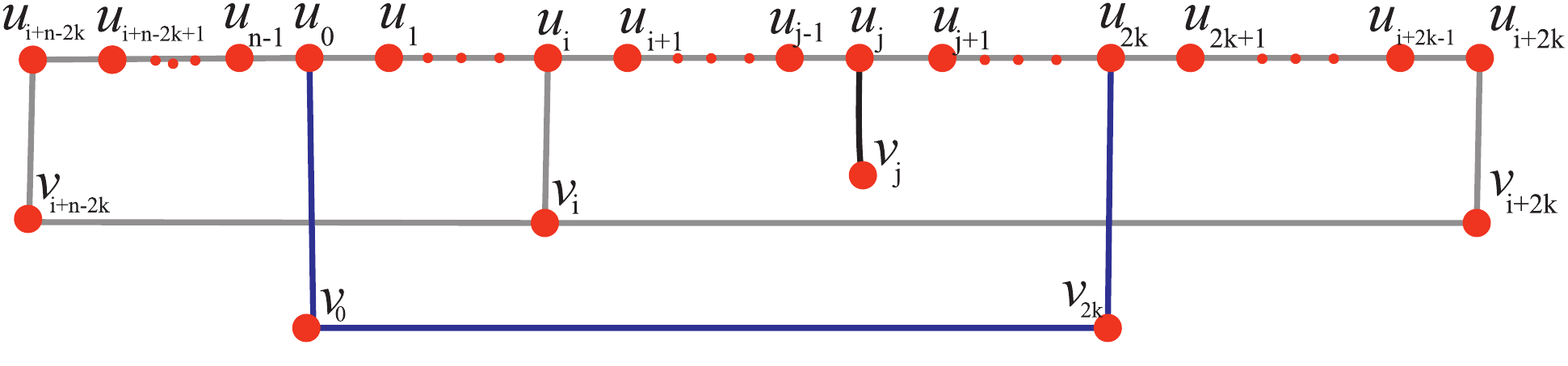}}
\caption{A $2k+3$- cycle in $\matr{Pet}(n,k)$.} \label{petcomp}
\end{figure}
}\end{proof}

 We will need the following result to prove Corollary~\ref{cor:compk}$(b)$.
\begin{alphthm}{{\rm \cite{yang}}\label{pro:zhu} Let $\overline{\matr{K}}_{_{\frac{p}{q}}}$ be the complement of the circular complete graph $\matr{K}_{_{\frac{p}{q}}}$. Then, for $\frac{p}{q}\geq 2$, we have
$$\chi_{_c}(\overline{\matr{K}}_{_{\frac{p}{q}}})=\frac{p}{\lfloor \frac{p}{q}\rfloor}.$$
}\end{alphthm}

 \begin{proof}{(of Corollary~\ref{cor:compk}$(b)$)\\
Since $2k+1 < \matr{Pet}(n,2k)$, by Theorem~\ref{lem:haji} we have,
$$\left(\overline{\matr{K}}_{_{\frac{2n}{4k+2}}}\right)^{^{\frac{1}{2k+1}}}\rightarrow \matr{Pet}(n,2k)\quad \Longleftrightarrow \quad \overline{\matr{K}}_{_{\frac{2n}{4k+2}}}\rightarrow \matr{Pet}(n,2k)^{^{2k+1}}.$$
Theorem~\ref{thm:girthbounds}$(b)$ implies that
$\overline{\matr{K}}_{_{\frac{2n}{4k+2}}}\rightarrow \matr{Pet}(n,2k)^{^{2k+1}}$, and consequently,
$$\left(\overline{\matr{K}}_{_{\frac{2n}{4k+2}}}\right)^{^{\frac{1}{2k+1}}}\rightarrow \matr{Pet}(n,2k).$$
But since $n>2k$ by Theorem~\ref{pro:zhu}, we have
$$\chi_{_c}\left( \matr{Pet}(n,2k)\right)\geq
\chi_{_c}\left(\left(\overline{\matr{K}}_{_{\frac{2n}{4k+2}}}\right)^{^{\frac{1}{2k+1}}}\right)=\frac{2n(2k+1)}{2kn+\lfloor \frac{2n}{4k+2}\rfloor}.$$
}\end{proof}
In \cite{ghebleh}, the following lower bound has been obtained for the circular chromatic number of generalized Petersen graphs.
\begin{alphthm}{\label{thm:gheblelower}{\rm \cite{ghebleh}} For any $n>2k$, we have
$$\chi_{_c}(\matr{Pet}(n,2k))\geq 2+\frac{2}{2k+1}.$$
}\end{alphthm}
The following proposition shows that the lower bound of Theorem~\ref{thm:homk2ko}$(b)$ can be strictly greater than the lower bound presented in Theorem~\ref{thm:gheblelower}.

 \begin{pro}{\label{pro:circular}
Let $n\stackrel{2k+2}{\equiv} y$ and $0<y\leq 2k+1- \dfrac{n}{2k+2}$, then
$$\frac{2n(2k+1)}{2kn+ \lfloor \frac{n}{2k+1}\rfloor} > 2+\frac{2}{2k+1} $$
}\end{pro}
\begin{proof}{
Since $y\leq 2k+1-\dfrac{n}{2k+2}$, if we set $\oo=\lfloor \dfrac{n}{2k+2}\rfloor$, then
$$n=\oo(2k+2)+y \Rightarrow n=\oo(2k+1)+(\oo+y).$$
Since $2k+2\nmid n$ we have $\oo< \dfrac{n}{2k+2}$, and consequently,
$$\oo+y<(\dfrac{n}{2k+2})+(2k+1- \dfrac{n}{2k+2})=2k+1.$$
Therefore,
$\lfloor \dfrac{n}{2k+1} \rfloor =\ \oo\ = \lfloor \dfrac{n}{2k+2} \rfloor$. Finally,
$$\frac{2n(2k+1)}{2kn+ \lfloor \frac{n}{2k+1}\rfloor}=
\frac{2n(2k+1)}{2kn+ \lfloor \frac{n}{2k+2}\rfloor}\geq
\frac{2n(2k+1)}{2kn+ \frac{n}{2k+2}}=
2+\frac{2}{2k+1}.$$
But, equality holds if and only if $2k+2|n$, which is impossible
since $n\stackrel{2k+2}{\equiv} y$ and $y\neq 0$.
}\end{proof}
It must be noted that there exist many pair of numbers $(n,k)$, such that
$4k^2-1<n<4k^2+6k+2$
for which conditions of both Proposition~\ref{pro:circular} and Corollary~\ref{cor:compk}$(b)$ hold simultaneously
(e.g. $5\leq n<12$ for $\matr{Pet}(n,2)$, or $16\leq n<30$ for $\matr{Pet}(n,4)$).

 %Also one of the purposes that make Proposition~\ref{pro:circular} and Corollary~\ref{cor:compk}$(b)$ important, is that %sometimes one may conjecturing that the circular chromatic number of generalized Petersen graphs are equal to the lower bound of %Proposition~\ref{pro:gheb}.
\begin{proof}{(of Theorem~\ref{thm:uphom})
\begin{itemize}
\item[$(a)$] Assume that $n$ is odd, $k$ is even and $n\stackrel{k-1}{\equiv} -2$,
(for $n\stackrel{k-1}{\equiv} 2$ one may use the same argument on the reverse direction for $\matr{Pet}(n,k)$).
Also, let
$$V(\matr{Pb}(n,k)) = \{x_{_{0}},x_{_{1}},\cdots,x_{_{n-1}}\}.$$

 Note that if $j(k-1)=n+2$, then $$\frac{k}{2}j(k-1)=\frac{k}{2}n+k\stackrel{n}{\equiv}k ,\quad (\frac{k}{2}j-1)(k-1)=\frac{k}{2}n+1\stackrel{n}{\equiv}1 \qquad (1),$$
and consequently, since $n$ and $(k-1)$ are relatively prime, we can define
$$\sigma(x_{_{i(k-1)}})\isdef y_{_{i}},$$
where operations are in $\mathbb{Z}_{_{n}}$.

 By the hypothesis,
$s=\frac{(n-4)(k-2)}{2(k-1)}$ and
$$\frac{(n-4)(k-2)}{2(k-1)}=s \leq \frac{k(n+2)}{2(k-1)} \ \Rightarrow \ \frac{k(n+2)}{2(k-1)}-1\leq n-s=n-\frac{(n-4)(k-2)}{2(k-1)}.$$
Hence, using $(1)$ and symmetry, one may verify that the edge $x_{_0}x_{_1}=x_{_0}x_{_{(\frac{k}{2}j-1)(k-1)}}$ is mapped to $y_{_{0}}y_{_{\frac{k}{2}j-1}}=y_{_{0}}y_{_{\frac{k(n+2)}{2(k-1)}-1}} \in E(\matr{K}_{_{\frac{n}{s}}})$, and also,
$x_{_0}x_{_k}=x_{_0}x_{_{\frac{k}{2}j(k-1)}}$ is mapped to $y_{_{0}}y_{_{\frac{k}{2}j}}=y_{_{0}}y_{_{\frac{k(n+2)}{2(k-1)}}}
\in E(\matr{K}_{_{\frac{n}{s}}})$.
\item[$(b)$] Again, using Watkins' notation for $\matr{Pet}(n,k)$, we can define
$$\sigma_{_V}(v_{_i})\isdef \sigma(u_{_{i+1}})\isdef x_{_i},$$
which is a homomorphism.
\item[$(c)$] Here, let
$$V(\matr{C}_{_{n}}^{^k}) = \{x_{_{0}},x_{_{1}},\cdots,x_{_{n-1}}\},$$ and define
$\sigma(v_{_i})\isdef \sigma_{_V}(u_{_{i+1}})\isdef x_{_i}.$
Consider the following three cases.
\begin{itemize}
\item Edges of type $u_{_i}u_{_{i+1}}$, are mapped to $x_{_{i+1}}x_{_{i+2}}$,
\item Edges of type $v_{_i}v_{_{i+k}}$ are mapped to $x_{_{i}}x_{_{i+k}}$,
\item Edges of type $v_{_{i}}u_{_i}$ are mapped to $x_{_{i}}x_{_{i+1}}$,
\end{itemize}
which proves that $\sigma$ is a homomorphism.
\end{itemize}
}\end{proof}

 \begin{proof}{(of Corollary~\ref{cor:upbound})
\begin{itemize}
\item[$(a)$] By Theorem~\ref{thm:uphom}$(b)$,
$$\matr{Pet}(n,k) \rightarrow \matr{Pb}(n,k),$$
and by Theorem~\ref{thm:uphom}$(a)$, for $s=\frac{(n-4)(k-2)}{2(k-1)}$ we have a homomorphism
$$\matr{Pb}(n,k) \rightarrow \matr{K}_{_{\frac{n}{s}}}.$$
Thus, there exists a homomorphism
$$\matr{Pet}(n,k) \rightarrow \matr{K}_{_{\frac{n}{s}}},$$
and consequently,
$$\chi_{_c}(\matr{Pet}(n,k))\leq \chi_{_c}( \matr{K}_{_{\frac{n}{s}}})=\frac{2n(k-1)}{(k-2)(n-4)}.$$
\item[$(b)$]
First, one may verify that $\chi_{_c}(\matr{C}_{_{n}}^{^k})=\frac{2n}{n-k},$ by considering the following homomorphism $\eta: \matr{C}_{_{n}}^{^k} \rightarrow \matr{K}_{_{\frac{n}{\frac{n-k}{2}}}}$,
$$\eta_{_V}(x_{_i})=\left\lbrace
\begin{array}{ll}
y_{_{\frac{i}{2}}}\qquad & i{\rm \ is\ even}\\
y_{_{\frac{n+\ii}{2}}}\qquad& i{\rm \ is\ odd},
\end{array}
\right.$$
in which $V(\matr{K}_{_{\frac{n}{\frac{n-k}{2}}}}) = \{y_{_{0}},y_{_{1}},\cdots,y_{_{n-1}}\}.$

 Now, by Theorem~\ref{thm:uphom}$(c)$, there exists a homomorphism
$$\sigma:\ \matr{Pet}(n,k) \rightarrow \matr{C}_{_{n}}^{^k}.$$
Thus, we have
$$\chi_{_c}(\matr{Pet}(n,k))\leq \chi_{_c}(\matr{C}_{_{n}}^{^k})=\frac{2n}{n-k}.$$
\end{itemize}
}\end{proof}

 \begin{proof}{(of Corollary~\ref{cor:ogirthclosed})
\begin{itemize}
\item[{\rm (a)}] Since $k$ is even, if $\go(\matr{Pet}(n,k))$ tends to infinity, then by Theorem~\ref{thm:girthbounds}, $k$ tends to infinity, and consequently, if $n$ is odd and $n\stackrel{k-1}{\equiv}\pm 2$, by Corollary~\ref{cor:upbound} the circular chromatic number $\chi_{_c}(\matr{Pet}(n,k))$ tends to $2$ because $\frac{2n(k-1)}{(n-4)(k-2)}$ tends to 2.
\item[{\rm (b)}] If both $n$ and $k$ are odd and $n\geq 5k$, by Corollary~\ref{cor:upbound} we have,
$$\chi_{_c}(\matr{Pet}(n,k))\leq \frac{2n}{n-k}\leq \frac{5}{2}.$$
\end{itemize}
}\end{proof}

 \end{document}